\documentclass[11pt]{article}

\usepackage{amsmath, amssymb, eucal, latexsym, multicol, amsthm, bm, placeins}
\usepackage{color}

\usepackage{latexsym}
\usepackage{amssymb}
\usepackage{amsmath}
\usepackage{amsfonts}
\usepackage{amsthm}
\usepackage{stmaryrd}
\usepackage{graphicx}
\usepackage{enumerate}
\usepackage{geometry}
\usepackage{bm}
\usepackage{comment}
\usepackage{ulem}

\usepackage{xspace}

\makeatletter\@addtoreset{equation}{section}\makeatother

\newcommand{\rsnote}[1]{{\color{cyan}{[rs: #1]}}}

\newtheorem{theorem}{Theorem}[section]
\newtheorem{remark}{Remark}[section]
\newtheorem{lemma}{Lemma}[section]

\newtheorem{proposition}{Proposition}[section]

\newtheorem{corollary}{Corollary}[section]

\newtheorem{property}{Property}[section]

\newcommand{\cT}{\mathcal{T}}
\newcommand{\cK}{\mathcal{K}}

\newcommand{\nrm}{| \hspace{-.6mm} | \hspace{-.6mm} |}

\newcommand{\cD}{\mathcal{D}}
\newcommand{\cP}{\mathcal{P}}

\newcommand{\cM}{\mathcal{M}}
\newcommand{\cO}{\mathcal{O}}

\newcommand{\eps}{\varepsilon}

\newcommand{\HPAFEM}{$\mathbf{hp}$-{\bf AFEM}\xspace}

\newcommand{\Solve}{{\bf SOLVE}}
\newcommand{\Estimate}{{\bf ESTIMATE}\xspace}
\newcommand{\Mark}{{\bf  MARK}\xspace}
\newcommand{\Refine}{{\bf  REFINE}\xspace}

\newcommand{\Hpnearbest}{$\mathbf{hp}$-{\bf NEARBEST}\xspace}
\newcommand{\Pde}{{\bf PDE}\xspace}
\newcommand{\Reduce}{{\bf  REDUCE}\xspace}
\newcommand{\Reproj}{{\bf\tt  REPROJECT}\xspace}
\newcommand{\Hpafem}{$\mathbf{hp}$-{\bf AFEM}\xspace}

\newcommand{\bb}{b}
\newcommand{\BB}{B}
\newcommand{\n}{\nu}

\DeclareMathOperator*{\argmin}{\rm{argmin}}

\newcommand{\osc}{{\rm{osc}}}

\newcommand{\R}{\mathbb R}
\newcommand{\cE}{{\mathcal E}}
\newcommand{\cN}{{\mathcal N}}
\newcommand{\cC}{{\mathcal C}}
\newcommand{\cL}{{\mathcal L}}

\DeclareMathOperator{\diam}{diam}

\DeclareMathOperator{\est}{\mathcal{E}}
\DeclareMathOperator{\errtotal}{\mathfrak{E}}

\DeclareMathOperator{\err}{E}

\newcommand{\tvert}{\vert\!\vert\!\vert}

\def \cT {{\mathcal{T}}}
\def \cM {{\mathcal{M}}}
\def \HOm {{H^1_0(\Omega)}}
\def \HI {{H^1_0(K_D)}}
\def \HmOm {{H^{-1}(\Omega)}}

\newenvironment{algotab}%
{\par\begin{samepage}%
\begin{tabbing}\ttfamily%
 \hspace*{5mm}\=\hspace{3ex}\=\hspace{3ex}\=\hspace{3ex}\=\hspace{3ex}%
\=\hspace{3ex}\=\hspace{3ex}\=\hspace{3ex}\=\hspace{3ex}\kill}%
{\end{tabbing}\end{samepage}}
\newcommand{\wt}{\widetilde}
\newcommand{\wh}{\widehat}

\title{{\bf Convergence and Optimality of \HPAFEM}}
\author{
Claudio~Canuto\thanks{Dipartimento di Scienze Matematiche,
Politecnico di Torino,
Corso Duca degli Abruzzi 24,
I-10129 Torino, Italy (claudio.canuto@polito.it )} 
\and
Ricardo~H.~Nochetto%
\thanks{Department of Mathematics and Institute for Physical Science
and Technology, University of Maryland, College Park, MD, USA
(rhn@math.umd.edu)}
\and Rob~Stevenson%
\thanks{Korteweg-de Vries Institute for Mathematics,
University of Amsterdam,
P.O. Box 94248,
1090 GE Amsterdam, The Netherlands
(r.p.stevenson@uva.nl)}
\and Marco~Verani\thanks{MOX-Dipartimento di Matematica, Politecnico di Milano, P.zza Leonardo Da Vinci 32, I-20133 Milano, Italy (marco.verani@polimi.it).}}
\begin{document}
\maketitle

\abstract{
We design and analyze an adaptive $hp$-finite element method
(\HPAFEM) in dimensions $n=1,2$.
The algorithm consists of iterating two routines: 
\Hpnearbest finds a near-best $hp$-approximation of the current
discrete solution and data to a desired accuracy, and
\Reduce improves the discrete solution to a finer but comparable accuracy.
The former hinges on a recent algorithm by  Binev for adaptive
$hp$-approximation, and acts as a coarsening step. We prove
convergence and instance optimality. 
}

%%%%%%%%%%%%%%%%%%%%%%%%%%%%%%%%%%%%%%%%%%%%%%%%%%%%%%%%%%%%%%%%%%%%%%%%%%%%%%%%
\section{Introduction}
%%%%%%%%%%%%%%%%%%%%%%%%%%%%%%%%%%%%%%%%%%%%%%%%%%%%%%%%%%%%%%%%%%%%%%%%%%%%%%%%

The discovery that elliptic problems with localized singularities,
such as corner singularities, can be approximated with exponential
accuracy propelled the study and use of $hp$-FEMs, starting with the
seminal work of Babu\v ska.
The a priori error analysis originated in the late 70's with the earliest 
attempts to study the adaptive approximation of a univariate function, 
having a finite number of singularities and otherwise smooth, 
by means of piecewise polynomials of variable degree \cite{DaSc:79,DVSc:80}.
These results influenced Gui and Babu\v ska
in their pioneering study of the convergence rate of the
$hp$-approximation to a one dimensional model elliptic problem in
\cite{GuiBa:86b}  and in their subsequent work \cite{GuiBa:86c}, which
proves convergence of an adaptive $hp$-algorithm with a predicted rate. However, due to the assumptions on the admissible error estimators, which appear to be overly restrictive, the results in \cite{GuiBa:86c} cannot be considered completely satisfactory. Starting from the late 80's the study of a posteriori error estimators and the design of adaptive $hp$-algorithms has been the subject of an intense research. We refer to the book \cite{SCHW98}, and the survey paper \cite{CaVe:13}, as well as the references therein for more details.

However, despite the interest in $hp$-FEMs, the study of adaptivity 
is much less developed than for the $h$-version of the FEM, for which a rather
complete theory has been developed in the last decade
\cite{Doerfler:96,MNS:00,BDD04,CKNS:08,Stevenson:07}; we
refer to the survey \cite{NoSiVe:09}. Regarding the $hp$-FEM,
we mention \cite{ScSi:00,DoHe:07,BuDo:11,BaPaSa:14} which prove
convergence without rates.
The purpose of this paper is to bridge this gap: we present a new 
\HPAFEM, which hinges on a recent algorithm by Binev for adaptive
$hp$-approximation \cite{Bi:13,Bi:14}, and prove several properties
including instance optimality in dimensions $n=1,2$. The theory is
complete for $n=1$ but there are a couple of pending issues for $n=2$,
which we discuss below.

The success of \HPAFEM 's hinges on having solutions and data
  with suitable sparsity structure, as well as practical algorithms
  that discover such a structure via computation. This is why
  existing \HPAFEM software typically
  probes the current discrete
  solution to learn about the local smoothness
of the exact solution, 
  but can only search around the current level of
  resolution. We refer, in particular, to the algorithms presented
  in \cite{Mav:94,AinsworthSenior:1998,Houston-Senior-Suli:2002,HoustonSuli:2005,MelenkEibner:2007} for strategies based on
  analyticity checks or local regularity estimation (see also
  \cite{ScSi:00,DoHe:07}), to the algorithms in
  \cite{DemkoviczOden-III:1989,DemkoviczOden-II:1989,DemkoviczOden-I:1989,DemkowiczRachowicz:2002} and \cite{TexasThreeSteps} for strategies
  based on the use of suitable discrete reference solutions, and to the algorithm in \cite{MW01} for a strategy based on comparing estimated and predicted errors.

%--------------------------------------------------------------------------------
\subsection{Challenges of $hp$-Approximation}
%--------------------------------------------------------------------------------

To shed light on the difficulties to design \HPAFEM,
we start with the much simpler problem of {\it $hp$-approximation} for
$n=1$. Let $\Omega := (0,1)$ and $K$ be a dyadic interval obtained
from $K_0=\bar\Omega$. Let $p$ be the polynomial degree associated with
$K$ at a certain stage of the adaptive algorithm, and denote
$D=(K,p)$. Given $v\in L^2(\Omega)$ and $p\ge0$, let
\begin{equation}\label{}
e_D(v):= \min_{\varphi \in \mathbb{P}_p(K)} \Vert v-\varphi\Vert_{L^2(K)}^2 
\quad \text{and} \quad
Q_D(v):=\argmin_{\varphi \in \mathbb{P}_p(K)} \Vert v-\varphi\Vert_{L^2(K)},
%v_D:=\argmin e_D(v).
\end{equation}
the latter function being extended with zero outside $K$.
The following algorithm generates a sequence of $hp$-decompositions
$({\cD}_\ell)_{\ell=0}^\infty$ and corresponding piecewise polynomial approximations 
$v_\ell=v_{\cD_\ell}$. With $v_0 := Q_{K_0,0}(v)$,  for $\ell>0$ and any $D=(K,p) \in \cD_\ell$,
\begin{enumerate}[$\bullet$]
\item 
compute $e_{K, p+1}(v-v_\ell)$ as well as $e_{K',p}(v-v_\ell)$ and 
$e_{K'',p}(v-v_\ell)$ for $K'$ and $K''$ being the two children of $K$;

\item if
$
e_{K, p+1} (v-v_\ell) < e_{K',p}(v-v_\ell) + e_{K'',p}(v-v_\ell),
$
then replace $D$ by $\wt D :=(K, p+1)$ in $\cD_{\ell+1}$ and set
$$
v_{\ell+1} := v_{\ell} +  Q_{\wt D}(v-v_\ell);
$$ 

\item otherwise, replace $D$ by $D' :=(K', p)$ and
  $D'' :=(K'', p)$ in $\cD_{\ell+1}$ and set
$$
v_{\ell+1} := v_{\ell} + Q_{D'}(v-v_\ell)
+ Q_{D''}(v-v_\ell).
$$ 

\end{enumerate}
Although this algorithm is deliberately very rudimentary so as to simplify the
discussion, it mimics existing schemes that query whether it is
more advantageous to refine the element $K$ or increase the polynomial
degree $p$ by a fixed amount, say $1$. We wonder
whether such an algorithm may lead to near-optimal $hp$-partitions.
In order to elaborate on this question, 
we now present two extreme examples that illustrate the role of sparsity
for the design of \HPAFEM.

\medskip\noindent
{\bf Example 1: Lacunary Function.}
For a given integer $L>0$, let $v$ be a polynomial of degree 
$p := 2^L-1$, such that, on each dyadic interval $K$ of
generation $0\le\ell<L$, $v$ is $L^2$-orthogonal to the linear
polynomials with vanishing mean.
Since we need to impose $2^{\ell}$ orthogonality relations for each level $\ell$,
we get altogether $1 + 2 + 2^2 + \cdots + 2^{L-1} = 2^L - 1$ constraints.
We thus realize that a nontrivial polynomial of degree $p$ does exist 
because it has $2^L$ parameters. We also see that  the
algorithm above bisects all dyadic elements $K$ starting from $K_0$
until reaching the level $L$, and that $v_\ell$ for all
$0\le\ell<L$  is the piecewise constant function that takes the mean-value of $v$ on each element in $\cD_\ell$.
Even if the algorithm stops refining at level $L$
and chooses from then on to raise the polynomial degree by 1 in each of
the $p$ elements, then at least $p$ new degrees of freedom have to be
added in each interval to represent $v$ exactly. This leads to a total of
$p^2$ degrees of freedom activated to capture a polynomial of degree
$p$, thereby proving that this process is non-optimal. We conclude that
to be near-optimal, \HPAFEM must be able to backtrack and review
decisions made earlier. This process, from now on called {\it coarsening},
is missing in most algorithms for $hp$-adaptivity except, for example, that of 
Demkowicz, Oden and Rachowicz \cite{DemkoviczOden-III:1989},
for which there are no optimality
results. The preceding function is extremely sparse for
$hp$-approximation, in fact a single polynomial, but its structure
is hard to discover in practice because of the sparsity gap.

\medskip\noindent
{\bf Example 2: Non-degenerate Function.} 
We now consider the canonical function 
$v(x) = x^{\alpha}$ with $\alpha<1$ on $\Omega = (0,1)$,
studied by DeVore and Scherer \cite{DVSc:80} and by Gui and Babu\v ska
\cite{GuiBa:86b}, which does not exhibit a sparsity gap. In fact,
the following non-degeneracy property is valid: there exist constants
$C_1,C_2$ such that for all intervals $K$ and polynomial degrees $p$
\[
C_2 \le \frac{e_{K,p+1}(v)}{e_{K,p}(v)} \le C_1.
\]
The exponential rate of convergence derived a priori in \cite{GuiBa:86c}, as
well as the linear increase of polynomial degrees starting from the
origin, are based on this crucial property. Similar results have
  been derived later for $n=2$ by Babu\v ska and Guo \cite{GuoBabuska-I:1986,GuoBabuska-II:1986} and for $n=3$ by Schotzau, Schwab and Wihler \cite{SchotzauSchwabWihler:2012-I,SchotzauSchwabWihler:2012-II}; see \cite{SCHW98} for a
  thorough discussion of the cases $n=1,2$. It is thus conceivable, as
observed in practice, that decisions made by \HPAFEM's with a
building block such as that above do not
produce unnecessary degrees of freedom for problems such as Example 2.
The lack of a coarsening step in most existing $hp$-software could thus be
attributed to the very special geometric features of point and edge
singularities, this being a special rather than a universal
property to design an optimal \HPAFEM.

%---------------------------------------------------------------------------------
\subsection{Our contributions}
%---------------------------------------------------------------------------------

Since we wish to account for a large class of functions
(solutions and data), perhaps exhibiting degeneracies such as in Example 1,
 our \HPAFEM includes a coarsening routine, which we envisage to be
 unavoidable for obtaining optimality. 
Our \HPAFEM hinges on two routines, \Hpnearbest and \Reduce,
and the former in turn relies on the adaptive $hp$-approximation routine by Binev
\cite{Bi:13,Bi:14}. 
To describe them, let 
$u =u(f,\lambda) \in H^1_0(\Omega)$ be the solution to a 
second order elliptic PDE on a domain $\Omega\subset\R^n, n=1, 2$,  
with data $(f,\lambda)$, where $f$ denotes forcing
term(s) and $\lambda$ parameters such as coefficients.

Given a reduction factor $\varrho \in (0,1)$, a {\it conforming} $hp$-partition
$\cD$, (discontinuous) $hp$-FEM approximations
$(f_\cD,\lambda_\cD)$ to $(f,\lambda)$ over $\cD$, 
%and underlying {\it continuous} $hp$-Galerkin approximation $u_\cD$ over $\cD$ to the exact solution $u(f_\cD,\lambda_\cD)$ with data $(f_\cD,\lambda_\cD)$, the call
%
%\[
%[\bar\cD,\bar u_\cD] = \Reduce (\varrho,\cD,f_\cD,\lambda_\cD)
%\]
%
the routine \Reduce produces a {\it conforming} $hp$-refinement $\bar{\cD}$ such that  {
the $|\cdot|_{H^1}$-error 
%the error, in $\vert \cdot \vert_{H^1(\Omega)}$, 
in the ({\it continuous}) $hp$-fem Galerkin approximation on $\bar{\cD}$ to the exact solution $u(f_\cD,\lambda_\cD)$
is less than $\rho$ times the same Galerkin error relative to the partition $\cD$.}
%the routine \Reduce produces a {\it conforming} $hp$-refinement $\bar{\cD}$ such that, in $|\cdot|_{H^1(\Omega)}$,  the error in the ({\it continuous}) $hp$-fem Galerkin approximation on $\bar{\cD}$ to the exact solution $u(f_\cD,\lambda_\cD)$, so with data $(f_\cD,\lambda_\cD)$,
%is less than $\rho$ times this error on $\cD$.
This routine will be implemented as an AFEM routine that applies under a no-data-oscillation assumption.
%produces a {\it conforming} $hp$-refinement $\bar{\cD}$ of $\cD$ and a 
%{\it continuous} $hp$-Galerkin approximation $\bar u_\cD$ over $\bar{\cD}$ to
%$u(f_\cD,\lambda_\cD)$ such that
%
%\[
%|u(f_\cD,\lambda_\cD) - \bar{u}_\cD|_{H^1(\Omega)} \le \varrho
%|u(f_\cD,\lambda_\cD) - u_{\cD}|_{H^1(\Omega)}.
%\]

The routine \Hpnearbest deals with nonconforming meshes
and subordinate discontinuous functions.
Given a tolerance $\eps>0$, a generic function $v \in H^1(\Omega)$, and data
$(f,\lambda)$, \Hpnearbest produces a {\it nonconforming} $hp$-partition $\cD$ and
suitable projections $(f_\cD,\lambda_\cD)$ of the data onto discontinuous 
$hp$-FEM spaces over $\cD$. The output is such that the square root of a {specific} {\it error functional}
%involving the broken $hp$-energy error for $v$ and $hp$-data oscillation over $\cD$ 
is less than $\eps$. 
This error functional is defined as the sum of the squared broken $|\cdot|_{H^1}$-error in the best (discontinuous) $hp$-approximation over $\cD$ to $v$ and
$\delta^{-1}$ times the {squared} $hp$-data oscillation 
{$\osc_\cD^2(f,\lambda)$} over $\cD$, 
for a sufficiently small penalty parameter $\delta>0$.
{In turn, $\osc_\cD^2(f,\lambda)$ measures the errors $f-f_\cD$ and $\lambda-\lambda_\cD$ on the 
partition $\cD$ in such squared local norms, that the following bound, expressing  
the continuous dependence on data of the underlying linear problem, holds:
\begin{equation}\label{stab}
|u - u(f_\cD,\lambda_\cD)|_{H^1(\Omega)} \lesssim \osc_\cD(f,\lambda).
\end{equation}
}
The procedure  \Hpnearbest  is based on
Binev's algorithm and is {\it instance optimal} for this
functional.

Our algorithm \HPAFEM consists of a repetition of calls of \Hpnearbest and \Reduce
with decreasing error tolerances.
The calls of \Hpnearbest, with $v$ being the current approximation to the  solution $u$,
%approximation to the {\it perturbed} solution $u(f_\cD,\lambda_\cD)$, 
are made to guarantee {\it instance optimality} of the coarsened approximations.
Coarsening, however, increases the error by a constant
factor. This must be compensated by a judicious choice of
the reduction factor $\varrho$ of \Reduce so that the concatenation 
of the two routines produces a converging sequence. To realize this
idea we must account for the following additional issues.

\medskip\noindent
{\bf Making meshes $hp$-conforming}:
After a call of \Hpnearbest, the generally
nonconforming $hp$-partition $\cD$ has to be refined to a conforming
one $\cC(\cD)$ so that it can serve as input for \Reduce. This is
obviously an issue for dimension $n=2$ but not for $n=1$, in which case
one can take $\cC(\cD)=\cD$. One may wonder whether the cardinality of
$\cC(\cD)$ can be bounded uniformly by that of $\cD$ for $n=2$. To see
that the answer is negative in general consider the following
pathological situation: a large triangle of $\cD$ with 
high polynomial degree is surrounded by small triangles with
polynomial degree 1. 
This is the reason why, without further assumptions on the structure
of the solution $u$, we cannot guarantee for $n=2$ an optimal balance
between the accuracy of the $hp$-approximations and the cardinality of the
$hp$-partitions at stages intermediate to consecutive calls of \Hpnearbest.
Resorting to a discontinuous \HPAFEM would cure
this gap at the expense of creating other difficulties.

\medskip\noindent
{\bf Making functions continuous}:
In order to quantify the reduction factor $\varrho$ of \Reduce we must be
able to compare the (broken) $H^1(\Omega)$-errors of the best {\it continuous}
and {\it discontinuous} $hp$-FEM approximations over $\cC(\cD)$.
We show that the former is bounded by the latter with 
a multiplicative constant which depends
logarithmically on the maximal
polynomial degree for $n=2$. This extends upon a recent result of Veeser for
the $h$-version of the FEM \cite{Vee12}. Such constant does not depend
on the polynomial degree for $n=1$. 
This construction is needed for the analysis of \HPAFEM only
but not its implementation.

\medskip\noindent
{\bf Dealing with a perturbed problem}:
When, preceding to a call of \Hpnearbest, the current (continuous) $hp$-approximation to $u$  has a tolerance $\eps$,
\Hpnearbest will be called with a tolerance $\eqsim \eps$ in order to guarantee optimality of the coarsened discontinuous $hp$-approximation.
In addition, \Hpnearbest produces new approximations $(f_\cD,\lambda_\cD)$ to the data to be used in the subsequent call of \Reduce. {The prescribed tolerance  ensures, in view of the definition of the error functional, that  $\osc_\cD(f,\lambda) \lesssim \sqrt{\delta}\, \eps$. Hence, concatenating with (\ref{stab}), we are guaranteed} that  
$|u-u(f_\cD,\lambda_\cD)|_{H^1(\Omega)} \lesssim \sqrt{\delta}\, \eps$.
The routine \Reduce approximates the solution $u(f_\cD,\lambda_\cD)$,
and so cannot be expected to produce an approximation to $u$ that is
more accurate than $u(f_\cD,\lambda_\cD)$. Therefore, in order to
obtain convergence of the overall iteration, 
the condition $|u-u(f_\cD,\lambda_\cD)|_{H^1(\Omega)} \leq \xi \eps$ 
is needed for some parameter $\xi \in [0,1)$, which is achieved by selecting the penalty parameter $\delta$ to be sufficiently small.

\medskip
The routine \Reduce will be implemented as an AFEM consisting of the usual loop over \Solve, \Estimate, \Mark, and \Refine.
For $n=1$, we construct an estimator that is reliable and discretely
efficient, uniformly in $p$.
Consequently, the number of iterations to achieve some fixed error
reduction $\varrho$ is independent on the maximal polynomial degree.

For $n=2$, we employ the residual-based a posteriori error estimator analyzed by
Melenk and Wohlmuth \cite{MW01}, which turns out to be $p$-sensitive. 
We show that in order to achieve a fixed error reduction,
it suffices to grow the number of iterations more than
quadratically with respect to the maximal polynomial degree.
This sub-optimal result is yet another reason for optimality
degradation at stages intermediate between two consecutive calls of \Hpnearbest.
Nevertheless, our result improves upon a recent one by
Bank, Parsania and Sauter \cite{BaPaSa:14}, which requires 
the number of iterations to be proportional to the fifth power of the
maximal polynomial degree.

Throughout this work, we assume that the arising linear systems are solved exactly.
To control the computational cost, optimal iterative solvers, uniformly in the polynomial degree would be required. We refer to \cite{BCCD:2014} for an example.

This work is organized as follows.
We present \HPAFEM within an abstract setting
in Sect. \ref{S1} and prove that it converges, and that the sequence of outputs of \Hpnearbest is {\it instance optimal}.
We give a brief description of Binev's algorithm in Sect.
\ref{sec:Binev}.
In Sect. \ref{S:1d}, we apply the abstract setting to the general
$1$-dimensional elliptic problem.
Finally, in Sect. \ref{S3}
we apply the abstract theory to the Poisson
equation in two dimensions.

The following notation will be used thoughout the paper.
By $\gamma \lesssim \delta$ we will mean that $\gamma$ can be bounded
by a multiple of $\delta$, independently of parameters which $\gamma$
and $\delta$ may depend on. Likewise, $\gamma \gtrsim \delta$ is
defined as $\delta \lesssim \gamma$, and $\gamma\eqsim \delta$ as
$\gamma\lesssim \delta$ and $\gamma \gtrsim \delta$.

%%%%%%%%%%%%%%%%%%%%%%%%%%%%%%%%%%%%%%%%%%%%%%%%%%%%%%%%%%%%%%%%%%%%%%%%%%%%%%%%
\section{An abstract framework} \label{S1}
%%%%%%%%%%%%%%%%%%%%%%%%%%%%%%%%%%%%%%%%%%%%%%%%%%%%%%%%%%%%%%%%%%%%%%%%%%%%%%%%

We now present the \HPAFEM in two steps. We first deal with an ideal
algorithm and later introduce a practical scheme including \Reduce. We
also discuss a possible realization of \Reduce.

%-------------------------------------------------------------------------------
\subsection{Definitions and assumptions} \label{def_ass}
%-------------------------------------------------------------------------------
On a domain $\Omega \subset \R^n$, we consider a, possibly, parametric PDE
\begin{equation}\label{eq:main-problem}
A_\lambda u=f.
\end{equation}
Here the forcing $f$ and the parameter $\lambda$ (representing,
  e.g., the coefficients of the operator) are taken from some spaces
$F$ and $\bar{\Lambda}$ of functions on $\Omega$, such that there
exists a unique solution $u=u(f,\lambda)$ living in a space $V$ of
functions on $\Omega$. We assume, for simplicity, that $V$ 
and $F$ are Hilbert spaces over $\R$.
%and that $\bar{\Lambda}$ is a set of a Banach space.

We assume that we are given an essentially disjoint initial partition of $\bar{\Omega}$ into finitely many (closed) subdomains (the `element domains'). We assume that for each element domain $K$ that we encounter, there exists a unique way in which it can be split into element domains $K'$ and $K''$, the `children' of $K$, such that $K=K' \cup K''$ and $|K' \cap K''|=0$.
The set $\mathfrak{K}$ % $\mathfrak{K}$ 
of all these element domains is therefore naturally organized as an
infinite binary `master tree', having as its roots the element domains
of the initial partition of $\bar{\Omega}$.
A finite subset of $\mathfrak{K}$ 
is called a subtree of the master tree when it contains all roots {and
for each element domain in the subset both its parent and its sibling are in the subset.}
The leaves 
% $\mathcal{L} (\mathcal{T})$ 
of a subtree 
%$\mathcal{T}$
form an essentially disjoint partition of $\bar{\Omega}$. The set of
all such `$h$-partitions' will be denoted as $\mathbb{K}$.
For ${\cK}, \wt{\cK} \in \mathbb{K}$, we call $\wt{\cK}$ a refinement
of ${\cK}$, and denoted as $\cK \leq \wt\cK$, when any $K \in
\wt{\cK}$ is either in ${\cK}$ or has an ancestor in ${\cK}$.

Our aim is to compute `$hp$-finite element' approximations to $u$, i.e., piecewise polynomial approximations, with variable degrees, w.r.t. partitions from $\mathbb{K}$. In order to do so, it will be needed first to replace the data $(f,\lambda)$ by approximations from finite dimensional spaces. For that goal as well, we will employ spaces of piecewise polynomials, with variable degrees, w.r.t. partitions from $\mathbb{K}$, as will be described next.

For all $K \in \mathfrak{K}$, {let $V_K,F_K,\Lambda_K$ be
(infinite dimensional) spaces}
of functions on $K$, such that for any ${\cK} \in \mathbb{K}$, it
holds that, possibly up to isomorphisms,
$$
V \subseteq \prod_{K \in {\cK}} V_K,
\quad 
F = \prod_{K \in {\cK}} F_K,
\quad 
\Lambda \subseteq  \prod_{K \in {\cK}} \Lambda_K \subseteq \bar{\Lambda}.
$$
Here $\Lambda$ is a subset of $\bar{\Lambda}$, which contains all the
parameters that will be allowed in our adaptive algorithm \HPAFEM, and, for simplicity,  has a Hilbert topology.
%\cmnote{We have relaxed the equality sign for the spaces $F$ and $\Lambda$ since it seems to us that this is needed to accommodate the one dimensional case with variable coefficients and data in $H^{-1}$.}
%\rsnote{For the moment, I restored the equality signs, because you later added a remark saying that the PDE also should have a solution for the enlarged data space, meaning, I think, that equally well you can enlarge $F$ and $\Lambda$ from the beginning.}
%\rsnote{I finally understood the problem with $\Lambda$ in the 1D application, but I still don't see a problem with $F$. What is wrong with taking it as $L^2(\Omega) \times L^2(\Omega)$?}
For all $(K,d) \in \mathfrak{K}\times \mathbb{N}$ 
%(where $ {\mathbb{N}_{>0}}:= \mathbb{N}\setminus \{0\}$) 
(hereafter $\mathbb{N}$ stands for the set of all {\it positive} natural numbers)
and $Z \in \{V,F,\Lambda\}$, we assume finite dimensional spaces $Z_{K,d} \subset Z_K$ of functions on $K$ such that $Z_{K,d} \subseteq Z_{K,d+1}$, 
%\cmnote{relax to $\subseteq$?}
$Z_{K,d} \subset Z_{K',d} \times Z_{K'',d}$,
and $Z \cap \bigcup_{\cK \in \mathbb{K}, d \in \mathbb{N}}
  \prod_{K\in\mathcal{K}} Z_{K,d}$ is dense in $Z$.
%\RHN{We denote the corresponding global and local norms by  $\|\cdot\|_Z$ and $\|\cdot\|_{Z_{K,d}}$, respectively.}
 
In applications, $V_{K,d}$ will be a space of polynomials of dimension $\eqsim d$.
 For instance, when $K$ is an $n$-simplex, $V_{K,d}$ may be chosen as ${\mathbb P}_{p}(K)$, where the associated polynomial degree $p=p(d)$ can be defined as the largest value in $\mathbb{N}$ such that 
 $\text{dim} \, {\mathbb P}_{p-1}(K)={n+p-1 \choose p-1} \leq d$. 
 This definition normalizes the starting value $p(1)=1$ for all $n \in \mathbb{N}$.
 Only for $n=1$, it holds that $p(d)=d$ for all $n \in \mathbb{N}$.
 
 Analogously, the spaces $F_{K,d}$ and $\Lambda_{K,d}$ will be selected as (Cartesian products of) spaces of polynomials as well, of degrees 
 equal to $p$ plus some constant in $\mathbb{Z}$.
 %related to $p$ by some additive constant belonging to $\mathbb{Z}$.
 
In the following, $D \in \mathfrak{K} \times \mathbb{N}$ will denote
an {{\it $hp$-element}: it is a pair $(K_D,d_D)$ consisting of an
  element domain $K_D \in \mathfrak{K}$, and an integer} 
$d_D \in \mathbb{N}$. We will write $Z_D = Z_{K_D,d_D}$.
 
 For all $D \in \mathfrak{K} \times \mathbb{N}$, we assume a projector $Q_{D}: V \times F \times \Lambda \rightarrow V_{D} \times F_{D} \times \Lambda_{D}$, and a local  error functional
$e_{D}=e_{D}(v,f,\lambda) \geq 0$, that, for  $(v,f,\lambda) \in V
  \times F \times \Lambda$, gives a measure for the (squared) distance
  between $(v|_{K_D},f|_{K_D},\lambda|_{K_D})$ and its local
  approximation $(v_{D},f_{D},\lambda_{D}):=Q_{D}(v,f,\lambda)$.
 We assume that this error functional is non-increasing under both `$h$-refinements' and `$p$-enrichments', in the sense that
\begin{equation} \label{1}
\begin{split}
e_{D'}+e_{D''} & \leq e_D \text{ when } K_{D'},\, K_{D''} \text{ are the children of } K_D, \text{ and } d_{D'}=d_{D''}=d_D;\\
e_{D'} & \leq e_D \text{ when } K_{D'}= K_{D} \text{ and } d_{D'} \geq d_D.
\end{split}
\end{equation}

 A collection $\cD=\{D=(K_D,d_D)\}$ of $hp$-elements is called an 
{\it $hp$-partition} provided $\cK(\cD):=\{K_D: D \in \cD\} \in
  \mathbb{K}$. The collection of all $hp$-partitions is denoted as $\mathbb{D}$.
For $\cD \in \mathbb{D}$,
we set the $hp$-approximation spaces
$$
Z_\cD:=\prod_{D \in \cD} Z_D, \quad (Z \in \{V,F,\Lambda\}),
$$
and define
$$
\# \cD:=\sum_{D \in \cD} d_D.
$$

In our applications, the quantity $\# \cD$ is proportional to the dimension of
$Z_\cD$, and $e_D(v,f_D,\lambda_D)$ is the sum of the squared best approximation error of $v|_{K_D}$ from $V_D$ in $|\cdot|_{H^1(K_D)}$ and $\delta^{-1} $ times {the square of} the local data oscillation.

For $\cD \in \mathbb{D}$, we set the {\it global error functional}
$$
{\err_\cD(v,f,\lambda):=\sum_{D \in \cD}  e_D(v,f,\lambda),}
$$
which is a measure for the (squared) distance between $(v,f,\lambda)$ and its projection
\begin{equation} \label{81}
%\RHN{(v_\cD,f_\cD,\lambda_\cD) :=}
\Big({\prod_{D \in \cD}v_D},\prod_{D \in \cD}f_D,\prod_{D \in\cD}
\lambda_D \Big) \in V_\cD \times F_\cD \times \Lambda_\cD.
\end{equation}
%\RHN{It is worth emphasizing that the triple $(v_\cD,f_\cD,\lambda_\cD)$ is discontinuous in general.}

For ${\cD}, \wt{\cD} \in \mathbb{D}$, we call $\wt{\cD}$ a refinement
of ${\cD}$, and write ${\cD} \leq \wt{\cD}$, when both 
$\cK(\cD) \leq \cK(\wt{\cD})$, and $d_{\wt{D}} \geq d_D$, for any $D \in \cD$, $\wt{D} \in \wt{\cD}$ with
$K_D$ being either equal to $K_{\wt{D}}$ or an ancestor of $K_{\wt{D}}$.
With this notation, \eqref{1} {is equivalent to} 
\begin{equation} \label{8}
\err_{\wt{\cD}}(v,f,\lambda) \leq \err_\cD(v,f,\lambda) \quad\forall \wt{\cD} \geq \cD.
\end{equation}

We will apply a finite element solver that generally operates on a subset $\mathbb{D}^c$ of the set of $hp$-partitions $\mathbb{D}$, typically involving a restriction to those $\cD \in \mathbb{D}$ for which the `$h$-partition' $\cK(\cD)$ is `conforming'.
We assume that there exists a mapping $\cC: \mathbb{D} \rightarrow \mathbb{D}^c$
such that
\begin{equation} \label{eq:c}
  \cC(\cD) \geq \cD \quad \forall \cD \in \mathbb{D}.
  \end{equation}
  
We emphasize that even for $\cD \in \mathbb{D}^c$, generally the space $V_\cD$ is not a subspace of $V$. Conforming subspaces, used e.g. in Galerkin approximations, are defined as 
\begin{equation} \label{100}
V_{\cD}^c:=V_{\cD} \cap V.
\end{equation}

With regard to \eqref{81}, we introduce the notation
\[
 f_\cD := \prod_{D \in \cD}f_D,
\qquad
\lambda_{\cD} := \prod_{D \in \cD}\lambda_D,
\]
but reserve the symbol $v_\cD$ to denote later a suitable near-best
approximation to $v\in V$ from $V_\cD^c$.

%but reserve the symbol $v_\cD$ to denote later a suitable near-best
%approximation to $v\in V$ from $V_\cD^c$.

%-----------------------------------------------------------------------------------
\subsection{A {basic} $hp$-adaptive finite element method} \label{S:2.2}
%-----------------------------------------------------------------------------------

 Our aim is for given $(f,\lambda) \in F \times \Lambda$ and $\eps>0$,
 to find $\cD$ with an essentially minimal $\# \cD$
 %{and}
such that {$\err_\cD(u(f,\lambda),f,\lambda) \leq \eps$.}
 We will achieve this by alternately improving either the efficiency or the accuracy of the approximation. 
 To that end, {we begin by considering a {\it basic} algorithm, which highlights the essential ingredients of a $hp$-adaptive procedure.} We make use of the two routines described below.
 The first routine is available and will be discussed in
 Sect.~\ref{sec:Binev}. 
Since we are not concerned with complexity now, existence of the second routine is a simple consequence of the density of the union of the $hp$-approximation spaces in $V$.
 
\begin{itemize}
\item $[{\cD},f_{\cD},\lambda_{\cD}]:=\ $\Hpnearbest$\!\!(\eps,v,f,\lambda)$\\[5pt]
The routine \Hpnearbest takes as input $\eps>0$, and $(v,f,\lambda) \in V\times F \times \Lambda$, and outputs $\cD \in \mathbb{D}$ as  well as $(f_\cD,\lambda_\cD)$ such that {$ \err_\cD(v,f,\lambda)^{\frac{1}{2}} \leq \eps$} and, for some constants $0< \bb
 \leq 1 \leq \BB$, $\# \cD \leq \BB \# \wh\cD$ for any $\wh\cD \in \mathbb{D}$ with 
{$\err_{\wh{\cD}}(v,f,\lambda)^{\frac{1}{2}} \leq \bb\eps$.} 

\item $[\bar{\cD},\bar{u}]:=\Pde(\eps,\cD,f_{\cD},\lambda_{\cD})$\\[5pt]
The routine \Pde takes as input $\eps>0$, $\cD \in \mathbb{D}^c$, and data $(f_\cD,\lambda_\cD) \in F_\cD \times \Lambda_\cD$.
It outputs $\bar{\cD} \in \mathbb{D}^c$ with $\cD \leq \bar{\cD}$
and $\bar{u} \in V_{\bar{\cD}}^c$
such that $\|u(f_\cD,\lambda_\cD)- \bar{u}\|_V \leq \eps$.
\end{itemize}
The input argument $v$ of \Hpnearbest will be the current approximation to $u(f,\lambda)$. In an `$h$-adaptive' setting, usually the application of such a routine is referred to as `coarsening'.
Since the data $(f_\cD,\lambda_\cD) \in F_\cD \times \Lambda_\cD$
of \Pde is discrete, it will be said to satisfy a {\it no-data-oscillation}
assumption w.r.t. $\cD$.

We make the following abstract assumptions concerning the relation between the error functional, the norm on $V$, the mapping $(f,\lambda) \mapsto u(f,\lambda)$, and the constant $\bb$ of \Hpnearbest.
We assume the existence of constants $C_1, C_2>0$ with
\begin{equation}\label{0}
C_1 C_2 <\bb,
\end{equation} 
such that 
%\RHN{the mapping $(f,\lambda) \mapsto u(f,\lambda)$ is  stable in $V$ and the error functional$\err_\cD(w,f,\lambda)^{\frac{1}{2}}$ is Lipschitz w.r.t. the first argument in $V$:}
%
\begin{gather} \label{2}
\|u(f,\lambda)-u(f_{\cD}, \lambda_{\cD})\|_V  \leq C_1 \inf_{w \in V} \err_\cD(w,f,\lambda)^{\frac{1}{2}} \qquad \forall \cD \in \mathbb{D},\,\forall(f,\lambda) \in F \times \Lambda,
\\ \label{3}
\sup_{(f,\lambda) \in F \times \Lambda} |\err_\cD(w,f,\lambda)^{\frac{1}{2}}-\err_\cD(v,f,\lambda)^{\frac{1}{2}}|  \leq C_2  \|w-v\|_V \qquad \forall\cD \in \mathbb{D},\, \forall v,w \in V.
\end{gather}
%
%\RHN{Let's motivate \eqref{0}-\eqref{3} in the context of our applications, namely
%  \eqref{eD-def} and \eqref{osc-def}. The latter with $C_2=1$ is a consequence of
%  the triangle inequality. Since
%
%\[
%\|u(f,\lambda)-u(f_{\cD}, \lambda_{\cD})\|_V  \leq C_* \osc_\cD(f,\lambda)
%= C_* \left( \sum_{D\in\cD} \osc_D(f,\lambda)^2 \right)^{\frac12},
%\]
%
%  with a constant $C_*$ that depends on the norms in
 % $(V,F,\Lambda)$, we deduce \eqref{2} with $C_1=\delta C_*$. Finally,
 % making the penalty parameter $\delta$ sufficiently small, we can achieve
%  \eqref{0}. }

The condition \eqref{3} means that $\err_\cD(w,f,\lambda)^{\frac{1}{2}}$ is Lipschitz w.r.t. its first argument. In our applications, we will verify this condition with $C_2=1$.
The condition \eqref{2} will be a consequence of the continuous
dependence \eqref{stab} of the solution on the data, and the fact 
 that the error functional will contain {the square of} a data oscillation 
$\osc_\cD(f,\lambda)$. Since this term is penalized by a factor $\delta^{-1}$, we will be able to ensure  \eqref{2} with $C_1 \eqsim \sqrt{\delta}$ which yields \eqref{0} by taking $\delta$ sufficiently small.
  
 \medskip 
Our {basic}  $hp$-adaptive finite element routine reads as follows. 
\begin{algotab}
\> \HPAFEM $\!\!(\bar{u}_0,f,\lambda,\eps_0)$ \\
\>\> \% Input: $(\bar{u}_0,f,\lambda) \in V \times F \times \Lambda$, {$\eps_0>0$} 
with $\|u(f,\lambda)-\bar{u}_0\|_V \leq \eps_0$.\\
%\>\> \% \cm{Output: ${\cD}_\varepsilon \in \mathbb{D}$, ${\bar\cD}_\varepsilon \in \mathbb{D}^c$ with ${\bar\cD}_\varepsilon \geq {\cD}_\varepsilon$, ${u}_\varepsilon\in V^c_{{\bar\cD}_\varepsilon}$ such that $\|u(f,\lambda)-u_\varepsilon\|_V\leq \varepsilon$.}\\
\>\> \% Parameters: $\mu,\omega>0$ such that 
$C_1 C_2<\bb(1-\mu)$, $\omega \in (\frac{C_2}{\bb},\frac{1-\mu}{C_1})$, $\mu\in(0,1)$.\\

\>\> {\tt for} $i=1,2,\ldots$ {\tt do}\\
\>\>\> $[{\cD}_i,f_{\cD_i},\lambda_{\cD_i}]:=$\Hpnearbest$\!\!(\omega \eps_{i-1},\bar{u}_{i-1},f,\lambda)$ \\
\>\>\> $[\bar{\cD}_i,\bar{u}_i]:=\Pde(\mu\eps_{i-1},\cC(\cD_i),f_{\cD_i},\lambda_{\cD_i})$ \\
\>\>\> $\eps_i:=(\mu+C_1 \omega)\eps_{i-1}$ \\
\>\> {\tt end do}
\end{algotab}

\noindent Note that $\bb \omega-C_2>0$, and that
$\eps_i= (\mu+C_1 \omega)^i \eps_0$, where $\mu+C_1 \omega<1$.

\begin{theorem} \label{th1}  Assuming \eqref{0}-\eqref{3}, for the sequences $(\bar{u}_i)$, $(\cD_i)$ produced in \HPAFEM, writing $u=u(f,\lambda)$, it holds that
\begin{equation}\label{th2.1:1}
\|u-\bar{u}_i\|_V \leq \eps_i \quad  \forall i \geq 0,\qquad 
{\err_{\cD_i}(u,f,\lambda)^{\frac{1}{2}}  \leq (\omega +C_2) \eps_{i-1}} \quad \forall i \geq 1,
\end{equation}
and
\begin{equation}\label{th2.1:2}
\# \cD_i \leq \BB \# \cD \quad\text{for any } \cD \in \mathbb{D} \text{ with } 
{\err_\cD(u,f,\lambda)^{\frac{1}{2}}  \leq (\bb \omega-C_2) \eps_{i-1}.}
\end{equation}
\end{theorem}

\begin{proof} The bound $\|u-\bar{u}_0\|_V \leq \eps_0$ is valid by
  assumption. For $i\ge1$, the tolerances used for \Hpnearbest and
  \Pde, together with \eqref{2}, show that
\begin{equation}\label{1st-stat}
\begin{aligned}
\|u-\bar{u}_i\|_V & \leq \|u(f_{\cD_i},\lambda_{\cD_i})-\bar{u}_i\|_V 
+\| u-u(f_{\cD_i},\lambda_{\cD_i})\|_V \\
& \leq \mu \eps_{i-1} + C_1 \err_{\cD_i}(\bar{u}_{i-1},f,\lambda)^{\frac12}
\leq (\mu+C_1 \omega)\eps_{i-1}=\eps_i.
\end{aligned}
\end{equation}
The first statement follows for all $i\ge0$. Using this and 
\eqref{3} implies the second assertion
\[
\err_{\cD_i}(u,f,\lambda)^{\frac{1}{2}}  
\leq \err_{\cD_i}(\bar{u}_{i-1},f,\lambda)^{\frac{1}{2}} 
+C_2\|u-\bar{u}_{i-1}\|_V \leq (\omega+C_2) \eps_{i-1} 
\qquad\forall i\ge1.
\]
Let $\cD \in \mathbb{D}$ with $\err_\cD(u,f,\lambda)^{\frac{1}{2}}  \leq (\bb \omega-C_2) \eps_{i-1}$.
Then, again by \eqref{3}, $\err_{\cD}(\bar{u}_{i-1},f,\lambda)^{\frac{1}{2}}
\leq \bb \omega \eps_{i-1}$
and so $\# \cD_i \leq \BB \# \cD$ because of the optimality
  property of \Hpnearbest.
  \end{proof}

The main result of Theorem~\ref{th1} can be summarized by saying that \HPAFEM is 
{{\it instance optimal}} for reducing $\err_\cD(u(f,\lambda),f,\lambda)$ over 
$\cD \in \mathbb{D}$.
Recall that in our applications, $\err_\cD(u(f,\lambda),f,\lambda)$ will be the sum of the squared best approximation error in $u$ from the nonconforming space $V_\cD=\prod_{D \in \cD} V_D$ in the broken $H^1$-norm and a {squared} data oscillation term penalized with a factor $\delta^{-1}$.

Additionally, Theorem~\ref{th1} shows linear convergence to $u$ of the sequence $(\bar{u}_i)$ of conforming approximations, in particular $\bar{u}_i \in V_{\bar{\cD}_i}^c$ where $\cD^c \ni \bar{\cD}_i \geq \cC(\cD_i)$.
Since $\eps_i= (\mu+C_1 \omega)^i \eps_0$, the infinite loop in \HPAFEM can be stopped to meet any desired tolerance.

The preceding algorithm \HPAFEM has the minimal
structure for convergence and optimality. Since the routine \Pde neither
exploits the current iterate nor work already done, we present a
practical \HPAFEM in Sect. \ref{cost} which replaces \Pde by \Reduce.

Finally in this subsection, we comment on the implications of the instance optimality result concerning class optimality.
For $N \in \mathbb{N}$, let
$\cD_N:=\argmin\{\err_{\cD}(u,f,\lambda)^{\frac{1}{2}}\colon \cD \in
\mathbb{D},\#\cD \leq N\}$ and let the best approximation error be
\[
\sigma_N:=\err_{\cD_N}(u,f,\lambda)^{\frac{1}{2}}.
\]
\begin{remark}[algebraic decay]\label{remark10}
\rm
If $\sigma_N$ decays {\it algebraically} with $N$, namely
$\sup_N N^s\sigma_N < \infty$,
then for the sequence $(\cD_i)$ produced in \HPAFEM, one infers that
$\err_{\cD_i}(u,f,\lambda)^{\frac{1}{2}}$ decays algebraically
with $\# \cD_i$ with the optimal rate:
$\sup_i (\#\cD_i)^{s} \err_{\cD_i}(u,f,\lambda)^{\frac{1}{2}} < \infty$. 
In other words, instance optimality implies algebraic class optimality.
\end{remark}

\begin{remark}[exponential decay]\label{R:expo-decay}
\rm
For $hp$-approximation, it is more relevant to consider an {\it
exponential} decay of $\sigma_N$, i.e., 
$\sup_N e^{\eta N^\tau} \sigma_N <\infty$ for some $\eta, \tau>0$. 
This is precisely the situation considered in \cite{CNV13,CNV14,CSV14} for adaptive Fourier or Legendre methods.

Let us asssume, for convenience, that $\sigma_N = C_{\#} e^{-\eta N^\tau}$
for some constant $C_{\#}$ and ignore in subsequent calculations that 
$N$ has to be an integer. In view of Theorem~\ref{th1}, let $N$ and $\eps_{i-1}$ be so related that 
$\sigma_N = (b \omega-C_2) \eps_{i-1}$
Since apparently $\# \cD_i \leq BN$ and $\err_{\cD_i}(u,f,\lambda)^{\frac{1}{2}}
\leq (\omega +C_2) \eps_{i-1}$, we deduce 
\[
\sup_i \Big(e^{\tilde{\eta} (\#\cD_i)^\tau} 
\err_{\cD_i}(u,f,\lambda)^{\frac{1}{2}}\Big) \leq \frac{C_{\#}(\omega+C_2)}{b \omega -C_2},
\]
with $\tilde{\eta} := B^{-\tau} \eta$.

On the other hand, we will see 
in Corollary~\ref{corol20} that the routine \Hpnearbest satisfies its
optimality conditions for {\it any} $\BB>1$, at the expense of
$\bb=\bb(\BB)\downarrow 0$ when $\BB \downarrow 1$. Moreover, as we have seen, in our applications we will be able to satisfy \eqref{0}--\eqref{3} for any $\bb>0$ by taking the penalization parameter $\delta$ small enough.
Therefore, we conclude that if $\sigma_N$ decays exponentially, 
characterised by parameters $(\eta,\tau)$, then so do the errors 
produced by \HPAFEM for parameters 
$(\tilde{\eta},\tau)$, where $\tilde{\eta}=B^{-\tau} \eta$ can be chosen arbitrarily 
close to $\eta$ (at the expense of increasing the supremum
value). This situation is much better than that encountered in 
\cite{CNV13,CNV14,CSV14}.
\end{remark}

%-----------------------------------------------------------------------------------
\subsection{The practical $hp$-adaptive finite element method} \label{cost}
%-----------------------------------------------------------------------------------

To render \HPAFEM more practical we replace the routine \Pde by \Reduce,
which exploits the work already carried out within \HPAFEM and reads

\begin{itemize}
\item $[\bar{\cD},\bar u]:=\Reduce(\varrho,\cD,f_\cD,\lambda_\cD)$\\[5pt]
The routine \Reduce takes as input a partition $\cD \in \mathbb{D}^c$,
data $(f_\cD,\lambda_\cD) \in F_\cD \times \Lambda_\cD$, and a desired error reduction factor $\varrho \in (0,1]$, 
and produces a conforming partition $ \bar{\cD}=\bar{\cD}(\cD,\varrho) \in \mathbb{D}^c$ with $\bar{\cD}\geq \cD$ and a function
$\bar{u} \in V_{\bar{\cD}}^c$ such that 
\begin{equation}\label{eq:contr}
\|u(f_\cD,\lambda_\cD)-\bar{u}\|_V \leq \varrho \inf_{v \in V_{\cD}^c}\|u(f_\cD,\lambda_\cD)-v\|_V .
%\quad \forall (f_\cD,\lambda_\cD) \in F_\cD \times \Lambda_\cD.
\end{equation}
\end{itemize}

Inside the practical \HPAFEM, the routine \Reduce will be called with as input partition the result of mapping $\cC:\cD \rightarrow \cD^c$ applied to the output partition of the preceding call of \Hpnearbest. In order to bound the right-hand side of \eqref{eq:contr}, we make the following assumption:
\begin{equation} \label{4}
{
\inf_{w \in V_{\cC(\cD)}^c} \|v-w\|_V \leq C_{3,\cD} \inf_{(f,\lambda)
  \in F \times \Lambda} 
\err_{\cD}(v,f,\lambda)^{\frac{1}{2}} \qquad \forall \cD \in \mathbb{D},\, \forall v\in V.
}
\end{equation}
In our applications, the infimum on the right-hand side reads as the squared error in the broken $H^1$-norm of the best approximation to $v$ from $V_\cD=\prod_{D \in \cD} V_D$. The left-hand side reads as the squared error in $H^1_0(\Omega)$ of the best approximation to $v$ from $V_{\cC(\cD)}^c=H^1_0(\Omega) \cap \prod_{D \in \cD} V_D$.
The constant $C_{3,\cD}$ should ideally be independent of
$\cD$. We will see in Sect.~\ref{S:1d} that this is the case for our application in
dimension $n=1$. However, for $n=2$ we will show in Sect.~\ref{S3}
that $C_{3,\cD}$ depends logarithmically on the largest polynomial
degree; this extends a result by A. Veeser \cite{Vee12}.

\medskip
Our practical $hp$-adaptive finite element routine reads as follows:

\begin{algotab}
\> \HPAFEM $\!\!(\bar{u}_0,f,\lambda,\eps_0)$ \\
\>\> \% Input: $(\bar{u}_0,f,\lambda) \in V \times F \times \Lambda$, {$\eps_0>0$} 
with $\|u(f,\lambda)-\bar{u}_0\|_V \leq \eps_0$.\\
%\>\> \% \cm{Output: ${\cD}_\varepsilon \in \mathbb{D}$, ${\bar\cD}_\varepsilon \in \mathbb{D}^c$ with ${\bar\cD}_\varepsilon \geq {\cD}_\varepsilon$, ${u}_\varepsilon\in V^c_{{\bar\cD}_\varepsilon}$ such that $\|u(f,\lambda)-u_\varepsilon\|_V\leq \varepsilon$.}\\
\>\> \% Parameters: $\mu,\omega>0$ such that 
$C_1 C_2<\bb(1-\mu)$, $\omega \in (\frac{C_2}{\bb},\frac{1-\mu}{C_1})$, $\mu\in(0,1)$.\\
\>\> {\tt for} $i=1,2,\ldots$ {\tt do}\\
\>\>\> $[{\cD}_i,f_{\cD_i},\lambda_{\cD_i}]:=$\Hpnearbest$\!\!(\omega \eps_{i-1},\bar{u}_{i-1},f,\lambda)$ \\
\>\>\> $[\bar{\cD}_i,\bar{u}_i]
:=\Reduce(\frac{\mu}{1+(C_1+C_{3,\cD_i})\omega},\cC(\cD_i),f_{\cD_i},\lambda_{\cD_i})$\\
\>\>\> $\eps_i:=(\mu+C_1 \omega)\eps_{i-1}$ \\
\>\> {\tt end do}
\end{algotab}

\begin{corollary}[convergence and instance optimality]\label{C:conv}
Assuming \eqref{0}-\eqref{3} and  \eqref{4}, the sequences $(\bar{u}_i)$, $(\cD_i)$ produced in the practical \HPAFEM above satisfy properties \eqref{th2.1:1} and \eqref{th2.1:2} in Theorem~\ref{th1}.  

%writing $u=u(f,\lambda)$, it holds that
%$$
%\|u-\bar{u}_i\|_V \leq \eps_i \quad(\forall i \geq 0),\qquad 
%{\err_{\cD_i}(u,f,\lambda)^{\frac{1}{2}}  \leq (\omega +C_2) \eps_{i-1}} \quad(\forall i \geq 1),
%$$
%and
%$$
%\# \cD_i \leq \BB \# \cD \quad\text{for any } \cD \in \mathbb{D} \text{ with } 
%{\err_\cD(u,f,\lambda)^{\frac{1}{2}}  \leq (\bb \omega-C_2) \eps_{i-1}.}
%$$

\end{corollary}

\begin{proof}
In view of the second part of the proof of Theorem~\ref{th1}, it is sufficient to prove that $\|u-\bar{u}_i\|_V \leq \eps_i$.
We argue by induction. If $\|u-\bar{u}_i\|_V \leq \eps_{i-1}$, which is valid for $i=1$, then, after the $i$th call of \Hpnearbest, \eqref{4} and \eqref{2} imply that 
\begin{equation}\label{initial-PDE}
\begin{aligned}
\inf_{v \in V_{\cC(\cD_i)}^c} & \|u(f_{\cD_i},\lambda_{\cD_i}) -v\|_V  \\
&\leq \|u-\bar{u}_{i-1}\|_V
+\inf_{v \in V_{\cC(\cD_i)}^c} \|\bar{u}_{i-1}-v\|_V +\|u-u(f_{\cD_i},\lambda_{\cD_i})\|_V
\\ 
& \leq \eps_{i-1}+C_{3,\cD_i} \err_{\cD_i}(\bar{u}_{i-1},f,\lambda)+C_1 \err_{\cD_i}(\bar{u}_{i-1},f,\lambda)
\\
& \leq
(1+(C_{3,\cD_i}+C_1) \omega) \eps_{i-1}.
\end{aligned}
\end{equation}
Consequently, after the subsequent call of \Reduce, it holds that $\|u(f_{\cD_i},\lambda_{\cD_i}) - \bar{u}_i\|_V \le \mu\eps_{i-1}$ according to 
\eqref{eq:contr}. This result combined with \eqref{1st-stat} shows that $\|u-\bar{u}_i\|_V\le\eps_i$.
\end{proof}

\begin{remark}[complexity of \HPAFEM]\label{R:complexity}
\rm
Let us consider the case that the constants $C_{3,\cD}$, defined in \eqref{4},
are insensitive to $\cD$, namely,
\begin{equation} \label{18}
 C_3:=\sup_{\cD \in \mathbb{D}} C_{3,\cD}<\infty.
  \end{equation}
This entails that the reduction factor
$\varrho_i=\frac{\mu}{1+(C_1+C_{3,\cD_i})\omega}$ of \Reduce satisfies
$\inf_i \varrho_i>0$. Additionally, suppose that, given a fixed $\varrho \in (0,1]$, \Reduce realizes \eqref{eq:contr} with
\begin{equation} \label{26}
\sup_{\cD \in \mathbb{D}^c} \frac{\#\bar{\cD}(\cD,\varrho)}{\# \cD} <\infty.
\end{equation}
If, furthermore,
\begin{equation} \label{19}
{C_4:=\sup_{\cD \in \mathbb{D}} \frac{\#\cC(\cD)}{\#\cD} < \infty,}
\end{equation}
then the sequences $(\cD_i)_i$ and $(\bar{\cD_i})_i$ produced in
\HPAFEM are so that $\# \bar{\cD}_i \lesssim \# \cD_i$.
In view of the optimal control over $\# \cD_i$, given by
Theorem~\ref{th1} and Corollary~\ref{C:conv}, we would have
optimal control over the dimension of any $hp$-finite element
space created within \HPAFEM. This ideal situation only happens in 
the one-dimensional case.
\end{remark}

%------------------------------------------------------------------------------------
\subsection{A possible realization of \Reduce} \label{reduce}
%------------------------------------------------------------------------------------

Let $A_\lambda \in \cL(V,V')$ for all $\lambda\in \bar\Lambda$
{and} define the associated continuous bilinear 
form $ a_\lambda(v,w):=\langle A_\lambda v , w \rangle$ for any $v,w\in V$, where $\langle \cdot, \cdot \rangle$ denotes the duality pairing between $V$ and $V^\prime$. We assume that $A_\lambda$ is symmetric, which is equivalent to the symmetry of the form $a_\lambda$. We furtherly assume that each $a_\lambda$ is continuous  and coercive  on $V$,  with continuity and coercivity constants $\alpha^* \geq \alpha_* >0$ independent of $\lambda\in \bar{\Lambda}$.
It is convenient to introduce in $V$ the energy norm $\tvert v \tvert_{\lambda}=\sqrt{a_\lambda(v,v)}$ associated with the form $a_\lambda$, which
satisfies $\sqrt{\alpha_*} \Vert v \Vert_V \leq \tvert v \tvert_\lambda \leq \sqrt{\alpha^*} \Vert v \Vert_V$ for all $v \in V$. Let $F \subset V'$.

Given $\cD \in \mathbb{D}$ and data $(f,\lambda) \in F \times \bar{\Lambda}$,  the (Galerkin) solution $u_\cD(f,\lambda) \in V_\cD^c$ of 
\begin{equation}\label{eq:galerkin}
a_{\lambda}(u_\cD(f,\lambda), v)=\langle f, v\rangle  \quad \forall v \in V_\cD^c
\end{equation}
is the best approximation to $u(f,\lambda)$ from $V_\cD^c$ in $\nrm \cdot \nrm_{\lambda}$.
In view of a posteriori error estimation, we will consider Galerkin solutions from $V_\cD^c$ only for data in $F_\cD \times \Lambda_\cD$, i.e., for data without data oscillation w.r.t.  $\cD$.

For $\cD \in \mathbb{D}^c$, $D \in \cD$, let us introduce local a posteriori error indicators
$$
{\eta_{D,\cD}: V_\cD^c \times F_\cD \times \Lambda_\cD \rightarrow [0,\infty),}
$$
which give rise to {the} global estimator
\begin{equation}\label{global-error-est}
{\est_\cD(v,f_\cD,\lambda_\cD):=\left(\sum_{D \in \cD} \eta_{D,\cD}^2(v,f_\cD,\lambda_\cD)\right)^{1/2}.}
\end{equation}
Given data $(f_\cD,\lambda_\cD)$ without data oscillation w.r.t. $\cD$, $\est_\cD(v,f_\cD,\lambda_\cD)$ will be used with $v=u_\cD(f_\cD,\lambda_\cD)$ as an estimator for the squared error in this Galerkin approximation to $u(f_\cD,\lambda_\cD)$.
It should not be confused with $E_\cD(v,f,\lambda)$, the latter being the sum of local error functionals $e_D(v,f,\lambda)$, that estimates the squared error in a projection on $V_\cD \times F_\cD \times \Lambda_\cD$ of $(v,f,\lambda) \in V \times F \times \Lambda$.

Given any $\cM \subset \cD$, it will be useful to associate the
estimator {restricted to $\cM$}
$$
\est_\cD(\cM,v,f_\cD,\lambda_\cD):=\left(\sum_{D \in \cM} \eta_{D,\cD}^2(v,f_\cD,\lambda_\cD)\right)^{1/2}.
$$
We assume that $\est_\cD$ satisfies the following assumptions:
\begin{itemize}
\item {\bf Reliability:}
For $\cD \in \mathbb{D}^c$, and $(f_\cD,\lambda_\cD) \in F_{\cD}
\times \Lambda_{\cD}$, {there holds}
\begin{equation} \label{23}
\|u(f_\cD,\lambda_\cD)-u_{\cD}(f_\cD,\lambda_\cD)\|_V \lesssim 
{\est_\cD(u_{\cD}(f_\cD,\lambda_\cD),f_\cD,\lambda_\cD).}
\end{equation}
%
%{If $(f_\cD,\lambda_\cD) \in F_{\cD}\times \Lambda_{\cD}$ we say that $(f_\cD,\lambda_\cD)$ satisfies a {\it no data oscillation} assumption.}
\item {\bf Discrete efficiency:} 
For $\cD \in \mathbb{D}^c$, $(f_\cD,\lambda_\cD) \in F_{\cD}
\times \Lambda_{\cD}$, and for any $\cM \subset \cD$,
there exists a $ \bar{\cD}(\cM)\in \mathbb{D}^c$ with   $\bar{\cD}(\cM) \geq \cD$ and $\# \bar{\cD}(\cM) \lesssim \# \cD$, such that 
\begin{equation}  \label{24}
\|u_{\bar{\cD}(\cM)}(f_\cD,\lambda_\cD)-u_{\cD}(f_\cD,\lambda_\cD)\|_V \gtrsim 
{\est_\cD(\cM,u_{\cD}(f_\cD,\lambda_\cD),f_\cD,\lambda_\cD).}
\end{equation}
\end{itemize}

Then a valid procedure \Reduce is defined as follows.
%\pagebreak
\begin{algotab}
\> $[\bar{\cD},u_{\bar{\cD}}]=\Reduce(\varrho,\cD,f_\cD,\lambda_\cD)$ \\
\>\> \% Input: $\varrho \in (0,1]$, $\cD \in \mathbb{D}^c$, $(f_\cD,\lambda_\cD) \in F_\cD \times \Lambda_\cD$.\\
\>\> \% Output: $\bar{\cD} \in \mathbb{D}^c$ with  $\bar{\cD} \geq \cD$, and the Galerkin solution $u_{\bar{\cD}}=u_{\bar{\cD}}(f_\cD,\lambda_\cD)$.\\
\>\> \% Parameters: $\theta \in (0,1]$ fixed.\\
\>\> {Compute {$M:=M(\varrho) \in \mathbb{N}$ sufficiently large, cf. Proposition~\ref{prop_reduce}.}}\\
\>\> $\cD_0:=\cD$;  \Solve: compute $u_{\cD_0}(f_\cD,\lambda_\cD)$\\
\>\> {\tt for $i=1$ to $M$ do} \\ 
\>\>\> \Estimate: {\tt compute}
$\{\eta^{{2}}_{D,\cD_{i-1}}(u_{\cD_{i-1}}(f_\cD,\lambda_\cD),f_\cD,\lambda_\cD)\colon
D \in \cD_{i-1}\}$ \\
\>\>\> \Mark: {\tt select $\cM_{i-1} \subseteq \cD_{i-1}$ with} \\ \\
\>\>\>\>
$
\est^2_{\cD_{i-1}}(\cM_{i-1},u_{\cD_{i-1}}(f_\cD,\lambda_\cD),f_\cD,\lambda_\cD)
  \geq
  \theta\est^2_{\cD_{i-1}}(u_{\cD_{i-1}}(f_\cD,\lambda_\cD),f_\cD,\lambda_\cD)
$
\\ \\
\>\>\> \Refine: $\cD_i:=\bar{\cD}(\cM_{i-1})$ \\
\>\>\> \Solve: compute $u_{\cD_i}(f_\cD,\lambda_\cD)$ \\
\>\> {\tt end} \\
\>\> $\bar{\cD}:=\cD_M$; $u_{\bar{\cD}}=u_{\cD_M}(f_\cD,\lambda_\cD)$
\end{algotab}

\vspace{0.5cm}
%\rhnnote{Wouldn't be better to write a while loop as in ICOSAHOM with
%stopping test dictated by the estimator?}

\begin{proposition} \label{prop_reduce} 
Assuming \eqref{23} and \eqref{24}, the number {$M=M(\varrho)$} of iterations that are required so that
$[\bar{\cD},u_{\bar{\cD}}(f_\cD,\lambda_\cD)]=\Reduce(\varrho,\cD,f_\cD,\lambda_\cD)$ satisfies
$$
\|u(f_\cD,\lambda_\cD)-u_{\bar{\cD}}(f_\cD,\lambda_\cD)\|_V \leq \varrho \inf_{v \in V_\cD^c} \|u(f_\cD,\lambda_\cD)-v\|_V
$$
is at most proportional to $\log \varrho^{-1}$, and $\# \bar{\cD} \lesssim \# \cD$, both independent of $\cD \in \mathbb{D}^c$, and $(f_\cD,\lambda_\cD) \in F_\cD \times \Lambda_\cD$.
So both \eqref{eq:contr} and \eqref{26} are realized.
\end{proposition}

\begin{proof} 
Since $f_\cD$ and $\lambda_\cD$ are fixed, for simplicity we
  drop them from our notations. 
Applying \eqref{24} with $\cD=\cD_{i-1}$ and $\cD_i=\bar\cD(\cM_{i-1})$, 
the definition of \Mark, and \eqref{23} we get 
\begin{equation*}
\begin{aligned}
\| u_{\cD_{i}}- u_{\cD_{i-1}}\|^2_V &\gtrsim
{\est^2_{\cD_{i-1}}(\cM_{i-1},u_{\cD_{i-1}}, f_{\cD},\lambda_\cD)} \\
&\geq \theta  {\est^2_{\cD_{i-1}}(u_{\cD_{i-1}}, f_{\cD},\lambda_\cD)} \\
&\gtrsim \theta \| u- u_{\cD_{i-1}}\|^2_V.
\end{aligned}
\end{equation*}
This and the uniform equivalence of $\|\cdot\|_V$ and $\nrm\cdot\nrm_{\lambda_\cD}=:\nrm\cdot\nrm$
give the {\it saturation property}
\begin{eqnarray}
\nrm u_{\cD_{i}}- u_{\cD_{i-1}}\nrm^2 \geq C_* \theta \nrm u- u_{\cD_{i-1}}\nrm^2
\end{eqnarray}
for some positive {constant} $C_*$. Then, using Pythagoras' identity
\begin{equation} \label{30}
\nrm u-u_{\cD_i}\nrm^2=\nrm u-u_{\cD_{i-1}}\nrm^2-
\nrm u_{\cD_i}-u_{\cD_{i-1}}\nrm^2,
\end{equation}
we obtain the {\it contraction property} 
\begin{equation}\label{contraction}
\nrm u-u_{\cD_i}\nrm\leq \kappa \nrm u-u_{\cD_{i-1}}\nrm
\end{equation}
for $\kappa=\sqrt{1-C_*\theta}<1$. We conclude that 
\begin{align*}
 \| u-u_{\cD_{M}} \|_{V} &\leq \frac{1}{\sqrt{\alpha_*}} 
\tvert u-u_{\cD_{M}} \tvert
\leq \frac{1}{\sqrt{\alpha_*}}  \kappa^M\tvert u-u_\cD \tvert \\
&=\frac{1}{\sqrt{\alpha_*}}  \kappa^M \inf_{v \in V_\cD^c}   \tvert u-v \tvert
\leq \sqrt{\frac{\alpha^*}{\alpha_*}} \kappa^M \inf_{v \in V_{\cD}^c}\|u-v\|_V.
\end{align*}
{ Enforcing $\sqrt{\frac{\alpha^*}{\alpha_*}} \kappa^M\leq \varrho$ yields $M={\cal O}(\log \varrho^{-1})$. In addition, since $\# \cD_i \lesssim \# \cD_{i-1}$ for $1\leq i \leq M$ according to \eqref{24}, the proof is complete.
}
\end{proof}

\begin{remark}
\rm The partition $\bar{\cD}(\cM)$ can be built by an `$h$-refinement'  or a `$p$-enrichment', or both, of the elements $D \in \cM$, if necessary followed by a `completion step' by an application of the mapping $\cC$ in order to land in $\mathbb{D}^c$.
The estimate $\# \bar{\cD}(\cM) \lesssim \# \cD$ shows no benefit in taking $\theta<1$, i.e., in taking a local, `adaptive' refinement.
In our algorithm \HPAFEM, the adaptive selection of suitable $hp$ partitions takes place in \Hpnearbest. Nevertheless, in a quantitative sense it can be beneficial to incorporate adaptivity in \Reduce as well, by selecting,  for a $\theta<1$, a (near) minimal set $\cM \subset \cD_{i-1}$ in \Mark.
\end{remark}

\begin{remark}
\rm 
{The discrete efficiency of the estimator implies its ``continuous'' efficiency.
Indeed, taking $\cM=\cD$ in \eqref{24} and denoting $\bar{\cD}=\bar{\cD}(\cD)$, and temporarily dropping $f_\cD$ and $\lambda_\cD$ from our notations, we have
\begin{align*}
\est_\cD(u_\cD)^2 &\lesssim \alpha_* \|u_{\bar{\cD}}-u_\cD\|_V^2 \leq  \tvert u_{\bar{\cD}}-u_\cD \tvert^2_\lambda = \tvert u-u_\cD \tvert^2_\lambda-\tvert u-u_{\bar{\cD}} \tvert^2_\lambda \leq \tvert u-u_\cD \tvert^2_\lambda\\
&=\inf_{v \in V_\cD^c} \tvert u-v \tvert^2_\lambda\leq \alpha^* \inf_{v \in V_\cD^c} \|u-v \|_V^2.
\end{align*}
Consequently, recalling \eqref{23}, a stopping criterium for \Reduce could be defined as follows 
$$ {\est_{\cD_i}(u_{\cD_i}(f_\cD,\lambda_\cD),f_\cD,\lambda_\cD)}\leq C \varrho {\est_\cD(u_{\cD}(f_\cD,\lambda_\cD),f_\cD,\lambda_\cD)},$$
where $C$ is a constant in terms of the ``hidden constants'' in \eqref{23} and \eqref{24}, and $\alpha_*$ and $\alpha^*$.}

%Using some of the arguments given in the proof of Proposition \ref{prop_reduce}, it is easily seen that the following bound holds for $\cD=\cD_0$ $${\est_\cD(u_{\cD}(f_\cD,\lambda_\cD),f_\cD,\lambda_\cD)}\lesssim  \inf_{v \in V_\cD^c} \|u(f_\cD,\lambda_\cD)-v\|_V.$$ Thus, recalling \eqref{23}, a stopping criterium for \Reduce could be defined as follows $$ {\est_{\cD_i}(u_{\cD_i}(f_\cD,\lambda_\cD),f_\cD,\lambda_\cD)}\lesssim \varrho {\est_\cD(u_{\cD}(f_\cD,\lambda_\cD),f_\cD,\lambda_\cD)}.$$

\begin{comment}
In view of the application of \Reduce to implement {\color{red} \Pde}, {alternatively one may} stop the iteration over $i$ as soon as the
upper bound provided by \eqref{23} for the error in the
Galerkin solution is less than or equal to the tolerance
prescribed in the call of {\color{red} \Pde}.
\end{comment}
\end{remark}

\medskip
Assumptions \eqref{23}-\eqref{24} about reliability and  discrete efficiency can be substituted
by the following three assumptions concerning the estimator.
{This will be used for our application in two dimensions in Sect.~\ref{S3}.}

\begin{itemize}
\item {\bf Reliability and efficiency:}
For $\cD \in \mathbb{D}^c$, there exists $R_\cD, r_\cD>0$, such that
for  $(f_\cD,\lambda_\cD) \in F_{\cD} \times \Lambda_{\cD}$, and
$\nrm\cdot\nrm_{\lambda_\cD}=:\nrm\cdot\nrm$ one has
\begin{equation} \label{260}
\begin{aligned}
r_\cD \est^2_\cD(u_{\cD}(f_\cD,\lambda_\cD),f_\cD,\lambda_\cD) &\leq
\nrm u(f_\cD,\lambda_\cD)-u_{\cD}(f_\cD,\lambda_\cD)\nrm^2 \\
& \leq R_\cD
\est^2_\cD(u_{\cD}(f_\cD,\lambda_\cD),f_\cD,\lambda_\cD);
\end{aligned}
\end{equation}

\item {\bf Stability:} {For $\cD \in \mathbb{D}^c$, and all $(f_\cD,\lambda_\cD) \in F_{\cD} \times \Lambda_{\cD}$, $v,w \in V^c_\cD$} one has
\begin{equation} \label{27}
\sqrt{r_\cD}\Big |{\est_\cD(v,f_\cD,\lambda_\cD)}-{\est_\cD(w,f_\cD,\lambda_\cD)}\Big|\leq \nrm v-w\nrm.
\end{equation}

\item {\bf Estimator reduction upon refinement:}
{
There exists a constant $\gamma<1$, such that 
for any $\cM \subset \cD \in \mathbb{D}^c$, there exists a $\bar{\cD}(\cM)\in \mathbb{D}^c$  with $\bar{\cD}(\cM) \geq \cD$, $\# \bar{\cD}(\cM) \lesssim \# \cD$, such that with 
$\bar{\mathcal{S}}:= \{\bar{D} \in \bar{\cD}(\cM):\exists D \in \cM \text{ with } K_{\bar{D}} \subset K_D\}$,
\begin{equation}  \label{28}
\begin{gathered}
{
\est^2_{\bar{\cD}(\cM)}(\bar{\mathcal{S}}, u_{\cD}(f_\cD,\lambda_\cD),f_\cD,\lambda_\cD) \leq \gamma 
\est^2_\cD(\mathcal{\cM},
u_{\cD}(f_\cD,\lambda_\cD),f_\cD,\lambda_\cD)
}
\\
{
\est^2_{\bar{\cD}(\cM)}( \bar{\cD}(\cM)\setminus \bar{\mathcal{S}}, u_{\cD}(f_\cD,\lambda_\cD),f_\cD,\lambda_\cD) \leq 
\est^2_\cD(\cD\setminus \mathcal{\cM}, u_{\cD}(f_\cD,\lambda_\cD),f_\cD,\lambda_\cD),
}
\end{gathered}
\end{equation}
for any $(f_\cD,\lambda_\cD) \in F_{\cD} \times \Lambda_{\cD}$.}
\end{itemize}
With $\theta$ from \Reduce and $\gamma$ from \eqref{28}, we set $\bar{\gamma}:=(1-\theta)+\theta \gamma$. For $\cD \leq \widehat\cD \in \mathbb{D}^c$, and $(f_\cD,\lambda_\cD) \in F_{\cD} \times \Lambda_{\cD}$, we define the squared {\it total error} to be
\[
\errtotal_{\widehat\cD}^2(u_{\widehat \cD}(f_\cD,\lambda_\cD),f_\cD,\lambda_\cD)
:= \nrm u(f_\cD,\lambda_\cD)-u_{\widehat\cD}(f_\cD,\lambda_\cD)\nrm^2
+ (1-\sqrt{\bar{\gamma}}) r_{\widehat\cD} \est^2_{\widehat\cD}(u_{\widehat\cD}(f_\cD,\lambda_\cD),f_\cD,\lambda_\cD).
\]

%{We now consider $\widehat\cD\ge\cD$ and scale the estimator by a factor $\xi$ and define the {\it total error} to be
%
%\[
%\errtotal_{\widehat\cD}^2(u_{\widehat \cD}(f_\cD,\lambda_\cD),f_\cD,\lambda_\cD)
%:= \nrm u(f_\cD,\lambda_\cD)-u_{\widehat\cD}(f_\cD,\lambda_\cD)\nrm^2
%+ \xi \est^2_{\widehat\cD}(u_{\widehat\cD}(f_\cD,\lambda_\cD),f_\cD,\lambda_\cD).
%\]
%}

\begin{proposition} \label{prop12} 
Assume \eqref{260}, \eqref{27}, and \eqref{28}, and, inside \Reduce, take $\bar{\cD}(\cM)$ as defined in \eqref{28}.
Let $\cD \in \mathbb{D}^c$, and $(f_\cD,\lambda_\cD) \in F_{\cD} \times \Lambda_{\cD}$.
%With $\gamma$ from \eqref{28}, and $\theta \in (0,1]$ from \Reduce, let $\bar{\gamma}:=(1-\theta)+\theta \gamma$.
%{If $\alpha_i := (1-\bar{\gamma})\frac{r_{\cD_i}}{R_{\cD_{i-1}}}$,
%and $\xi_i=(1-\sqrt{\bar{\gamma}})r_{\cD_i}$,
Then consecutive iterands produced in $\Reduce(\varrho,\cD,f_\cD,\lambda_\cD)$ satisfy
\begin{align*}
\errtotal_{\cD_i}^2(u_{\cD_i}(f_\cD,\lambda_\cD),f_\cD,\lambda_\cD)
\le 
{\Big[1-\frac{(1-\sqrt{\bar{\gamma}})^2}{2}\frac{r_{\cD_i}}{R_{\cD_{i-1}}}
\Big]}
\errtotal_{\cD_{i-1}}^2(u_{\cD_{i-1}}(f_\cD,\lambda_\cD),f_\cD,\lambda_\cD).
\end{align*}
Furthermore, for $\cD \in \mathbb{D}^c$ and $(f_\cD,\lambda_\cD) \in F_{\cD} \times \Lambda_{\cD}$, 
\begin{align*}
\nrm u(f_\cD,\lambda_\cD)-u_\cD(f_\cD,\lambda_\cD)\nrm^2 \leq
\errtotal_\cD^2(u_{\cD}(f_\cD,\lambda_\cD),f_\cD,\lambda_\cD)
\leq 
2 \nrm u(f_\cD,\lambda_\cD)-u_\cD(f_\cD,\lambda_\cD)\nrm^2.
\end{align*}
%with $\xi = (1-\sqrt{\bar{\gamma}})r_{\cD}$. 
Therefore,
if $\displaystyle{\sup_{\cD \in \mathbb{D}^c} R_\cD}<\infty$ and $\displaystyle{\inf_{\cD \in \mathbb{D}^c} r_\cD}>0$, then 
the statement of Proposition~\ref{prop_reduce} is again valid.
\end{proposition}

\begin{proof} {Since both $f_\cD$ and $\lambda_\cD$ are fixed, we
again drop them from our notations. Applying \Mark and \eqref{28}
yields} 
\begin{equation} \label{29}
{
\est^2_{\cD_i}(u_{\cD_{i-1}}) \leq \bar{\gamma}
\est^2_{\cD_{i-1}}(u_{\cD_{i-1}}).
}
\end{equation}
By virtue of \eqref{27}, Young's inequality, and \eqref{29}, 
we have that for any $\zeta>0$, 

\begin{align*}
\est^2_{\cD_i}(u_{\cD_i}) &\leq (1+\zeta)
\est^2_{\cD_i}(u_{\cD_{i-1}})+(1+\zeta^{-1}) r_{\cD_i}^{-1} \nrm
u_{\cD_i}-u_{\cD_{i-1}}\nrm^2
\\
&
\leq (1+\zeta) \bar{\gamma}
\est^2_{\cD_{i-1}}(u_{\cD_{i-1}})+(1+\zeta^{-1}) r_{\cD_i}^{-1} \nrm
u_{\cD_i}-u_{\cD_{i-1}}\nrm^2.
\end{align*}
By multiplying this inequality by $\frac{r_{\cD_i}}{(1+\zeta^{-1})}$, {substituting $\zeta=\bar{\gamma}^{-\frac{1}{2}}-1$},
and adding to Pythagoras' identity \eqref{30}, we obtain
\[
\nrm u-u_{\cD_i}\nrm^2+(1-\sqrt{\bar{\gamma}})r_{\cD_i} \est_{{\cD_i}}^2({\cD_i},u_{\cD_i}) \leq
\nrm u-u_{\cD_{i-1}}\nrm^2+ \sqrt{\bar{\gamma}}
(1-\sqrt{\bar{\gamma}})r_{\cD_i} \est_{{\cD_{i-1}}}^2({\cD_{i-1}},u_{\cD_{i-1}}).
\]
We resort to \eqref{260} to bound the right-hand side as
follows in terms of {an arbitrary} $\beta\in[0,1]$
\[
\beta \nrm u-u_{\cD_{i-1}}\nrm^2+\Big((1-\beta) \frac{R_{\cD_{i-1}}}{(1-\sqrt{\bar{\gamma}})r_{\cD_i}}+\sqrt{\bar{\gamma}}\Big) (1-\sqrt{\bar{\gamma}})r_{\cD_i} \est_{{\cD_{i-1}}}^2({\cD_{i-1}},u_{\cD_{i-1}}).
\]
We now observe that the following function of $\beta$ attains its
minimum at $\beta_*$
\[
\max_\beta \Big\{\beta,\Big((1-\beta)
\frac{R_{\cD_{i-1}}}{(1-\sqrt{\bar{\gamma}})r_{\cD_i}}+\sqrt{\bar{\gamma}}\Big)
\Big\} \ge
\beta_*
:= {1-\frac{1-\sqrt{\bar{\gamma}}}{1+\frac{R_{\cD_{i-1}}}{(1-\sqrt{\bar{\gamma}})r_{\cD_i}}}}.
\]
The proof of the first statement follows from  {$\frac{1-\sqrt{\bar{\gamma}}}{1+\frac{R_{\cD_{i-1}}}{(1-\sqrt{\bar{\gamma}})r_{\cD_i}}} \geq \frac{(1-\sqrt{\bar{\gamma}})^2}{2}\frac{r_{\cD_i}}{R_{\cD_{i-1}}}$}.
The second statement is a direct consequence of \eqref{260}, and the final statement follows directly from the first two.

\end{proof}

%\cmnote{We removed former Section $4$.}

%%%%%%%%%%%%%%%%%%%%%%%%%%%%%%%%%%%%%%%%%%%%%%%%%%%%%%%%%%%%%%%%%%%%%%%%%%%%%%
\section{{The module \Hpnearbest}} \label{sec:Binev}
%%%%%%%%%%%%%%%%%%%%%%%%%%%%%%%%%%%%%%%%%%%%%%%%%%%%%%%%%%%%%%%%%%%%%%%%%%%%%%

{In this section we describe briefly the algorithm and theory recently
developed by P. Binev for $hp$-adaptive tree approximation
\cite{Bi:14}, which constitutes the building block behind the module
\Hpnearbest.}
\begin{comment}
We first observe that for any given $h$-partition $\cK$ there is 
a unique binary subtree $\cT$ of $\mathfrak{K}$ whose leaves
form $\cK$. Since the cardinalities of $\cT$ and $\cK$ satisfy
%
\[
\#\cT = 2\#\cK-1,
\]
%
their computational complexities differ by a factor $2$. This
explains why, in contrast to \cite{BiDV04,Bi:14}, we focus below
on $\cK$ instead to $\cT$.
\end{comment}

%-----------------------------------------------------------------------------
\subsection{$h$-Adaptive Tree Approximation}\label{S:h-tree}
%-----------------------------------------------------------------------------

We first review the algorithm designed and studied by Binev and
DeVore \cite{BiDV04} for $h$-adaptive tree approximation.
% of one single function $v \in L^2(\Omega)$ \cite{BiDV04}.  
Since, in this subsection, the local approximation spaces do not depend on $d$, temporarily we identify an element $D$ with the element domain $K_D$, and $\cD$ with the {\it $h$-partition} $\cK(\cD)$, the latter being an element of $\mathbb{K}$.

Recall that for any $\cK \in \mathbb{K}$, the set of all $K \in \cK$ together with their ancestors form a tree $\cT$, being a subtree of the master tree $\mathfrak{K}$.
Conversely, given such a subtree $\cT$, the set $\cL(\cT)$ of its leaves is a partition in $\mathbb{K}$.

For the moment, we will assume that the master tree $\mathfrak{K}$ has only {\it one root}. In the next subsection, in Remark~\ref{multiple-roots}, we will deal with the case that it has possibly multiple roots.

For any $K\in\mathfrak{K}$, let $e_K \geq 0$ be some {\it local $h$-error functional}. That means that it satisfies the key property \eqref{1}, that in this $h$-element setting reduces to
%\cm {In the following, we set $D=(K_D, 1)$, where $K_D$ will denote
%an {\it $h$-element} and $\cK$ the corresponding $h$-partition}.  We let the 
%local $h$-error functional
%$e_{K_D,1} = \cm{e_{D}}(v)$ be the square of the best $L^2$-approximation
%error by affine polynomials.
{\it subadditivity}: 
\[
e_{K'} + e_{K''} \le e_{K}
\]
where $K'$ and $K''$ denote the children of $K$.
The corresponding {\it global $h$-error functional} reads
\[
\err_{\cK} = \sum_{K\in\cK} e_{K}
\quad\forall \cK \in \mathbb{K}.
\]
%\cm{which, with an abuse of notation, we associate to $\cK$ instead of $\cD$, in view of the fixed choice of the polynomial degree on each element $K_D$}.

The notion of a {\it best $h$-partition} w.r.t. this error functional is now apparent:
for $N \in \mathbb{N}$, let
%The notion of {\it best $h$-approximation} for $v\in L^2(\Omega)$ isnow apparent: 
%if $N_0\ge1$ denotes the number of roots of
% $\mathfrak{K}$, for $N\ge N_0$ let
%
\[
\sigma_{N} := \inf_{\#\cK\le N} \err_{\cK}.
\]
{This quantity} gives the smallest error achievable with $h$-partitions $\cK$ with
cardinality $\#\cK\le N$. In spite of the $\inf$ being a $\min$,
because the minimization is over a finite set, computing a tree that
realizes the $\min$ has exponential complexity.

A fundamental, but rather surprising, result of Binev and DeVore 
shows that a {\it near-best $h$-adaptive tree} is computable with
linear complexity. A key ingredient is a {\it modified local $h$-error functional}
$\tilde e_{{K}}$ defined as follows
for all $K\in\mathfrak{K}$:
\vskip0.3cm
\begin{itemize}
\item
$\tilde e_{K} := e_{K}$ if $K$ is the root;
\item
$\frac{1}{\tilde e_{K}} := \frac{1}{e_{K}} + \frac{1}{e_{K^*}}$
where ${K^*}$ is the parent of ${K}$ and $e_K \ne 0$; otherwise 
$\tilde e_K = 0$.
\end{itemize}
This harmonic mean has the following essential properties:
if $e_{K} \ll e_{K^*}$, then $\tilde e_{K}\approx
e_{K}$, whereas if $e_{K} \approx e_{K^*}$, then 
$\tilde e_{K}\approx \frac12 e_{K}$. {This means that
$\tilde e_{K}$ penalizes the lack of success in reducing the
error from $K^*$ to $K$ up to a factor $\frac12$, provided
$e_K=e_{K^*}$, and always
$\frac12\le \frac{\tilde e_{K}}{e_{K}} <1$.} 

%Note that $e_{K}$ does not depend on the tree $\cT$ to which $K$ belongs.

The practical method consists of applying a {\it greedy}
algorithm based on $\{\tilde e_{K}\}_{{K}\in\cK}$:
given an $h$-partition $\cK_N$, with $\#\cK_N = N$, construct $\cK_{N+1}$ by
bisecting an element domain $K\in\cK$ with largest $\tilde e_{K}$.
It is worth stressing that if lack of error reduction persists,
then the modified error functional $\tilde e_{K}$ diminishes
exponentially and
forces the greedy algorithm to start refining somewhere else.  

For $e_K$ being the squared {$L^2$-error} in the best polynomial
approximation on $K$ of a function $v$, this may happen when $v$
 has local but strong singularity.
 %with very high frequency.
The simple,
but astute idea to operate on the modified error functionals is responsible alone for the following key result.

\begin{theorem}[instance optimality of $h$-trees \cite{BiDV04}]\label{T:h-tree}
Let the master tree $\mathfrak{K}$ have one single root.
The sequence of $h$-partitions $(\cK_N)_{N \in \mathbb{N}}$ given by the greedy algorithm
based on 
$(\tilde e_{K})_{K\in\cK}$ provides near-best $h$-adaptive tree
approximations  in the sense that
\[
\err_{\cK_N} \le \frac{N}{N-n+1} \sigma_{n}
\quad {\forall n\le N.}
\]
The complexity  for obtaining $\cK_N$ is $\cO(N)$.
\end{theorem}

We can interpret Theorem \ref{T:h-tree} as follows:
given $N$ let $n = \lceil\frac{N}{2}\rceil$ be the ceiling of $N/2$, whence 
$N-n+1\ge N/2$ and
\begin{equation}\label{h-tree-est}
\err_{\cK_N} \le 2 \sigma_{\lceil\frac{N}{2}\rceil}.
\end{equation}
%

%-----------------------------------------------------------------------------
\subsection{$hp$-Adaptive Tree Approximation %: Single Root
}\label{S:hp-tree}
%-----------------------------------------------------------------------------
In this subsection, we return to $hp$-approximations.
An element $D$ is a pair $(K,d)=(K_D,d_D)$, with $K$ being the element domain, and $d$ an integer.
The local error functional $e_D \geq 0$ is required to satisfy \eqref{1}, i.e.,
$e_{K',d}+e_{K'',d} \leq e_{K,d}$ when $K', K''$  are the children of $K$,  and  
$e_{K,d'} \leq e_{K,d}$ when $d' \geq d$.
The corresponding global $hp$-error  functional reads as
$$
E_\cD=\sum_{D \in \cD} e_D \quad \forall \cD \in \mathbb{D}.
$$
For $N \in \mathbb{N}$, we set
$$
\sigma_N:=\inf_{\# \cD \leq N} E_{\cD}
$$
where $\# \cD=\sum_{D \in \cD} d_D$.

In our applications, $d_D$ is proportional to the dimension of the polynomial approximation space that is applied on $K_D$ so that
$\# \cD$ is proportional to the dimension of the global $hp$-finite element space.
More precisely, given $d$, we take $p=p(d)$ as the largest integer for which 
\begin{equation}\label{relation-pd}
\text{dim} \, {\mathbb P}_{p-1}(K)={n+p-1 \choose p-1} \leq d,
\end{equation}
and corresponding to $D=(K,d)$, we choose $ {\mathbb P}_{p(d)}(K)$ as approximation space. 
Consequently, for $n>1$, $e_{K,d+1}=e_{K,d}$ whenever $p(d+1)=p(d)$.

We describe an algorithm, designed by Binev \cite{Bi:13,Bi:14}, that finds a {\it near-best $hp$-partition}. It
 builds two
trees: a {\it ghost $h$-tree $\cT$}, 
similar to that in Sect. \ref{S:h-tree}
but with degree dependent error and modified error functionals,
and a {\it subordinate $hp$-tree $\cP$}. The second tree is obtained
by trimming the first one and increasing $d$
%the polynomial degree
as described in the sequel. 

%{For the sake of simplicity, we continue assuming that we wish to approximate a scalar function $v \in L^2(\Omega)$ using as local $hp$-error functional the square of the best $L^2$-approximation error by polynomials of chosen degree. The extension to vector-valued functions in Sobolev spaces, and to error functionals built by best approximation errors in the corresponding norms, poses no additional difficulty; as anticipated in Sect. \ref{def_ass}, the latter will be precisely the situation of interest for our applications.} 

%{It is convenient to distinguish between the cases in which the master $h$-tree $\mathfrak{K}$ has one single root or multiple roots.
%\subsubsection{The single root case}\label{S:hp-tree-1root}
%}
Let $\cK \in \mathbb{K}$, and let $\cT$ denote its corresponding tree.
For any $K\in\cT$, we denote by
$\cT(K)$ the {\it subtree} of $\cT$ emanating from $K$, and let
$d(K, \cT)$ be the number of leaves of $\cT(K)$, i.e.
\begin{equation}\label{dim-K}
d(K,\cT) = \# \cL(\cT(K)).
\end{equation}
%
%\RHN{This quantity dictates the admissible polynomial degree 
%$p(K,\cT)\ge1$ of $K$ within $\cT$, which plays no explicit role 
%in the procedure described below for growing the ghost $h$-tree:
%$d(K,\cT)$ will increase by $1$ unit at a time, which explains why it cannot be exactly the dimension of the polynomial space of degree $p(K,\cT)$ in dimension $n>1$. 
%Given $d=d(K,\cT)$, the degree $p=p(K,\cT)$ is defined as the largest integer for which 
%
%\begin{equation}\label{relation-pd}
%\text{dim} \, {\mathbb P}_{p-1}(K)={n+p-1 \choose p-1} \leq d.
%\end{equation}
%
%The use of ${\mathbb P}_{p-1}(K)$ instead of ${\mathbb P}_p(K)$
%normalizes the starting values of both $d$ and $p$, namely
%
%\[
%d = p = 1 \quad\forall n\ge 1.
%\]
% 
%In particular, for $p>1$, $d=p$ always for $n=1$ but for $n>1$ the value of $d$
%will increase while $p$ will remain constant until \eqref{relation-pd}
%will be valid with equality.
%}

The tree-dependent {\it local $hp$-error functionals $e_K(\cT)$} are defined recursively starting from the leaves and proceeding upwards as follows:
%To describe the construction of $\cP$ we start by denoting by
%${e_{K_D,d_D}=e_D(v)}$ the $L^2$ approximation error of function $v$ in
%$K_D$ with polynomials of degree $p$ {related to $d=d_D$ through} \eqref{relation-pd}; 
%this quantity is again independent of
%$\cT$ and is monotone in $p$. We next follow a bottom-top approach to define the 
%{\it local $hp$-error functionals} ${e_{K_D,d_D}}$: given an $h$-tree $\cT$, for each $K_D\in\cT$ we define ${e_{K_D,d_D}}=e_{K_D,d_D}(v,\cT)$  
%starting from the leaves
%$\cL(\cT)$ of $\cT$ and proceeding upwards as follows:
%
\begin{enumerate}[$\quad\bullet$]
\item
$e_K(\cT):= e_{K,1}$ provided $K\in\cL(\cT)$,

\item
{$e_K(\cT):= \min\{ e_{K'}(\cT) + e_{K''}(\cT), e_{K, d(K,\cT)} \}$} otherwise,
\end{enumerate}
where $K',K''\in\cT$ are the children of $K$. This local functional 
carries the information whether it is preferable to {enrich the
space (increase $d$) or refine the element (decrease $h$)}
to reduce the current error in $K$.
%However, in contrast to $e_{K_D,1}(v)$, {the quantity} $e_{K_D,d_D}$ depends on the tree $\cT$ where $K$ belongs through the quantity ${d(K_D,\cT)}$. 
%
The subordinate $hp$-tree $\cP$ is obtained from $\cT$ 
by eliminating {the subtree $\cT(K)$ of a node $K\in\cT$ whenever} 
\[
e_K(\cT)= e_{K, d(K,\cT)}.
\]
%
% This way all the $h$-degrees of freedom $d(K,\cT)$ associated to leaves of
% $\cT(K)$ are converted into $p$-degrees of freedom for element $K$,
% which becomes a leaf of $\cP$, 
% and $\cP$ turns out to be a minimal $hp$-tree with $p$-error 
% {$e_{K_D,d(K_D,\cT)}$} at all its leaves $K_D$.
This procedure is 
depicted  in Figure \ref{F:trees}.
% for spatial dimension $n=1$.
%
%\begin{figure}[h!]\label{F:trees}
%  \begin{center}
%    \includegraphics[width=2in]{figures/tree1.pdf}
%    \hskip1.cm
%    \includegraphics[width=1in]{figures/tree2.pdf}
%  \end{center}
  
  \begin{figure}[h!]\label{F:trees}
  \begin{center}
  \includegraphics[width=4in]{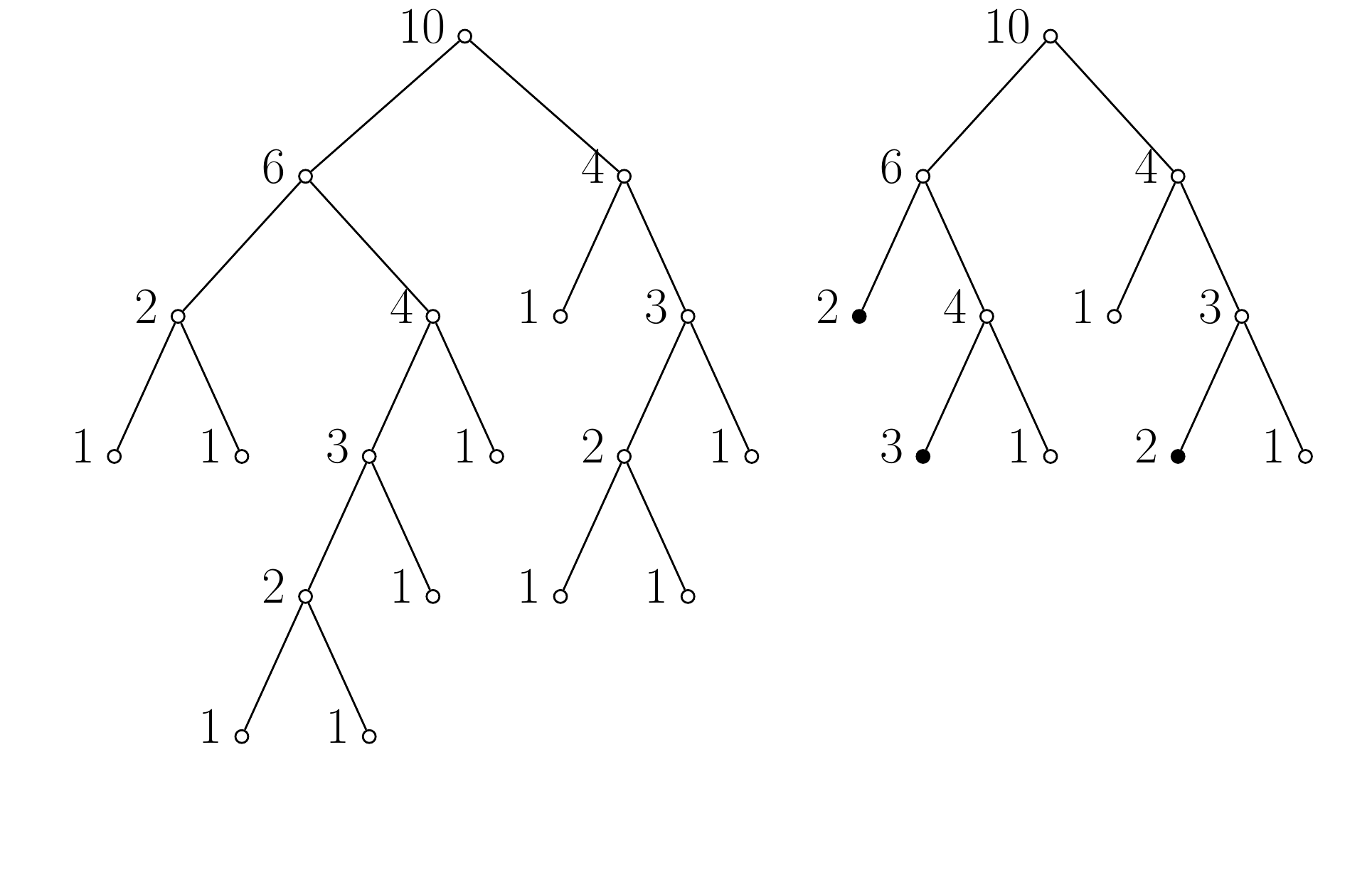}
    \end{center}

\vskip-0.5cm
\caption{\small Ghost $h$-tree $\cT$ (left) with 10 leaves ($\#\cL(\cT)=10$). The label of each node $K$ is $d(K,\cT)$. Subordinate
$hp$-tree $\cP$ (right) resulting from $\cT$ upon trimming 3 subtrees
and raising the values of $d$ of the interior nodes of $\cT$, now
leaves of $\cP$, from $1$ to $2, 3$, and $2$ respectively.
}
\end{figure}

The $hp$-tree $\cP$ gives rise to an $hp$-partition $\cD$, namely the
collection of $hp$-elements $D=(K,d)$ with $K$ a leaf of $\cP$
and $d = d(K,\cT)$.
We have that $\# \cD=\#\cK$, and $\cD$ minimizes $E_{\wt{\cD}}$ over
all $\wt{\cD} \in \mathbb{D}$ with $\cK(\wt{\cD}) \leq \cK$ and 
$d_D \le d(K_D,\cT)$ for all $D\in\wt\cD$,
whence $\# \wt{\cD} \leq \# \cK$.

This describes the trimming of the $h$-tree $\cT$, but not how to
increase the total cardinality of $\cT$. To grow $\cT$, P. Binev uses
a modified local $hp$-error functional and a greedy algorithm that
selects the leaf of $\cT$ that would lead to the largest reduction of
the $hp$-error in $\cP$. We refer to \cite{Bi:14} for the construction of 
the full algorithm for $hp$-adaptive approximation.

\begin{theorem}[instance optimality of $hp$-tree \cite{Bi:13,Bi:14}]\label{T:hp-tree}
Let the master tree $\mathfrak{K}$ have one single root.
For all $N\in\mathbb{N}$, the algorithm sketched above
constructs an $hp$-tree $\cP_N$ subordinate to a ghost $h$-tree $\cT_N$
such that the resulting $hp$-partition $\cD_N$ has cardinality
$\#\cD_N = N$ and global $hp$-error functional
\[
\err_{\cD_N} \le \frac{2N}{N-n+1} \sigma_{n}
\quad {\forall n\le N.}
\]
In addition, the cost of the algorithm for
obtaining $\cD_N$ is bounded by $\cO\big(\sum_{K\in\cT_N} d(K,\cT_N)\big)$, and
varies from $\cO(N\log N)$ for well balanced trees to $\cO(N^2)$ for
highly unbalanced trees.
\end{theorem}

Binev's algorithm gives a routine \Hpnearbest that satisfies the assumptions made in Subsect.~\ref{S:2.2} for any $B>1$ and $b=\sqrt{\frac{1}{2}(1-\frac{1}{B})}$:

\begin{corollary} \label{corol20} Let $B>1$. Given $\eps>0$, let $\cD \in \mathbb{D}$ be the first partition in the sequence produced by Binev's algorithm for which $E_\cD^{\frac{1}{2}} \leq \eps$. Then $\# \cD \leq B \min\{\# \hat{\cD}\colon \hat{\cD} \in \mathbb{D},\, \err_{\hat{\cD}}^{\frac{1}{2}} \leq \sqrt{\frac{1}{2}(1-\frac{1}{B})}\,\eps\}$.
\end{corollary}

\begin{proof} Let $\cD=\cD_N$, i.e.,  $\cD$ is  the $N$th partition in the sequence, and $\#\cD=N$. For $N=1$ the statement is true, so let $N>1$.
Suppose there exists a $\hat{\cD} \in \mathbb{D}$ with $\err_{\hat{\cD}}^{\frac{1}{2}} \leq \sqrt{\frac{1}{2}(1-\frac{1}{B})}\,\eps$ and $N>B \# \hat{\cD}$.
Then, with $n:=\# \hat{\cD}$, we have $\err_{\cD_{N-1}} \leq \frac{2(N-1)}{N-1-n+1} \sigma_n \leq \frac{2(N-1)}{N-1-n+1} \err_{\hat{\cD}} \leq  \frac{2(N-1)}{N-1-n+1} \frac{1}{2}(1-\frac{1}{B}) \eps^2$.
From $\frac{2(N-1)}{N-1-n+1} \frac{1}{2}(1-\frac{1}{B}) \leq 1$, being a consequence of $N \geq B n$ and $B \geq 1$, we get a contradiction with $\cD$ being the first one with $E_{\cD}^{\frac{1}{2}} \leq \eps$.
\end{proof}

\begin{remark} \label{multiple-roots}
In order to deal with the case that the master tree $\mathfrak{K}$ has $R>1$ roots, the following approach can be followed.

We unify the $R$ roots pairwise creating new element domains, each one being  the union of two roots.  When $R>2$, this process has to be repeated until only one element domain remains, which will the new, single root.
Obviously, this applies only when $R$ is a power of 2. In the other case, we have to introduce at most $\lceil \log_2 R \rceil -1$ (empty) virtual element domains (and, formally, infinite binary trees of virtual element domains rooted at them). 
We denote the extended, single rooted master tree by $\widehat{\mathfrak{K}}$.
%\footnote{\cm{We believe that $\widehat{\mathfrak{K}}$ is a better notation to suggest a completion of the tree.}}

Next, we extend the definition of $e_{K,d}$ as follows. At first  we give a meaning to $e_{K,0}$ 
for each element domain $K \in \mathfrak{K}$. Typically, for $d \in \mathbb{N}$, $e_{K,d}$ has the meaning of the squared error in the approximation of a quantity from a space of dimension $d$. Then a natural definition of $e_{K,0}$ is that of the squared error in the zero approximation.

Considering now the elements in $\widehat{\mathfrak{K}}\setminus {\mathfrak{K}}$, i.e., the newly created element domains, we distinguish between virtual and non-virtual element domains. For each virtual element domain, we set $e_{K,d}:=0$ for any $d \in \mathbb{N}\cup \{0\}$. Finally, for each newly created non-virtual element domain $K$, being the union of $K'$ and $K''$ (one of them possibly being a virtual element domain), for $d \in \mathbb{N} \cup \{0\}$ recursively we define 
$$
e_{K,d}:=\min_{\{d',d'' \in \mathbb{N} \cup \{0\}\colon d'+d''\leq d\}} e_{K',d'}+e_{K'',d''}.
$$
%\footnote{\cm{Although it is equivalent to Rob's original definition, we believe that using $d'+d''\cm{\leq} d$ makes easier to verify \eqref{1}}.}
Note that in the minimum at the right hand side $d'$ or $d''$ can or has to be zero.
In that case, $e_{K',d'}+e_{K'',d''}$ has the interpretation of the squared error in an approximation on $K$ that is zero on $K'$ or $K''$.

It is easily checked that the error functional $e_{K,d}$ for $(K,d) \in \widehat{\mathfrak{K}} \times \mathbb{N}$ satisfies \eqref{1}, and Theorem~\ref{T:hp-tree} and Corollary~\ref{corol20} apply.
\end{remark}

We close the discussion of the module \Hpnearbest with the observation that
in dimensions $n>1$, Binev's algorithm produces $hp$-partitions that are generally non-conforming. Since conformity is required by the module \Reduce, a post-processing step which makes the output partition conforming is required. 
The implementation of such a procedure in dimension 2, and the analysis of its complexity, will be {discussed} in Sect. \ref{2D-conformity}.

\section{A self-adjoint elliptic problem in 1D}\label{S:1d}
%%%%%%%%%%%%%%%%%%%%%%%%%%%%%%%%%%%%%%%%%%%%%%%%%%%%%%%%%%%%%%%%%%%%%%%%%%%%%%%%%%%%

In this section we apply the abstract framework introduced in Sect. \ref{S1} to a one-dimensional self-adjoint elliptic problem.

%------------------------------------------------------------------------------------
\subsection{The continuous problem and its $hp$ discretization}
%------------------------------------------------------------------------------------
Let $\Omega:=(0,1)$. Given $f_1,f_2\in L^2(\Omega)$ and $\nu,\sigma\in L^\infty(\Omega)$ 
satisfying 
\begin{equation}\label{hyp:data}
0< \nu_*\leq \nu\leq \nu^*<\infty  \qquad \text{and} \qquad 0 \leq \sigma\leq \sigma^*<\infty \;
\end{equation}
for some constants $\nu_*,\nu^*$ and $\sigma^*$, we consider the following model elliptic problem 
\begin{equation} \label{eq:two-point}
\begin{split}
 -&(\nu u^{\prime})^{\prime} +  \sigma u= f_1 + f_2^{\prime}\quad \text{in~}\Omega\,,
\\
& u(0)=u(1)=0\;,
\end{split}
\end{equation}
which can be written as in \eqref{eq:main-problem} setting $\lambda=(\nu,\sigma)$, $f=f_1+f_2^\prime \in H^{-1}(\Omega)$ and 
$$A_\lambda u:= -(\nu u^{\prime})^{\prime} +  \sigma u \in \mathcal{L}(H^1_0(\Omega), H^{-1}(\Omega)).$$
Equivalently, $u\in H^1_0(\Omega)=:V$, equipped with the norm $\vert \cdot \vert_{H^1(\Omega)}$, satisfies 
\begin{equation}\label{pb:weak}
a_\lambda(u,v) = \langle f , v \rangle  \qquad \forall v\in  H^1_0(\Omega),
\end{equation}
where the bilinear form $a_\lambda:H^1_0(\Omega)\times H^1_0(\Omega) \to \mathbb {R}$ and the linear form $f:H^1_0(\Omega) \to \mathbb{R}$ are defined as 
$$
a_\lambda(u,v):=\int_\Omega (\nu u^\prime v^\prime + \sigma u v) \, dx \;, \qquad \langle f , v \rangle=\int_\Omega (f_1v - f_2v^\prime) \, dx \;.
$$

%\rsnote{An option would be to give the variational formulation directly after \eqref{pb-var:1}-\eqref{pb-var:2}, and then to define $(A_\lambda u(v)=a_\lambda(u,v)$, i.e., to skip $A_\lambda$ in strong form}
 
In view of the approximation of the operator $A_\lambda$ we introduce the metric space 
$$\bar\Lambda:=\{ \bar\lambda=(\bar\nu,\bar\sigma)\in L^\infty(\Omega) \times L^\infty(\Omega): ~~
\bar\nu_*\leq \bar\nu\leq \bar\nu^*,~ -\bar\sigma_*\leq \bar\sigma\leq \bar\sigma^*  \} $$
where $\bar\nu_*, \bar\nu^*, \bar\sigma_*, \bar\sigma^*$ are positive constants defined as follows. 
Suppose that the pair $(\bar\nu,\bar\sigma)$ approximates $(\nu,\sigma)$ with error
\begin{equation}\label{eq:pert-bound}
\|\nu - \bar\nu\|_{L^\infty(\Omega)}\leq \frac{\nu_*}{2}, \qquad \|\sigma - \bar\sigma\|_{L^\infty(\Omega)}\leq \frac{\nu_*}{2};
\end{equation}
then it is easily seen that 
$$ 
\bar\nu_*:=\frac{\nu_*}{2}\leq \bar\nu\leq \nu^*+\frac{\nu_*}{2}=:\bar\nu^*,
\qquad -\bar\sigma_*:=-\frac{\nu_*}{2}\leq \bar\sigma\leq \sigma^*+\frac{\nu_*}{2}=:\bar\sigma^*.
$$
Furthermore, using the Poincar\'e inequality $\|v\|^2_{L^2(\Omega)}\leq \frac{1}{{2}}\vert v \vert^2_{H^1(\Omega)}$ we have
$$
(\bar\nu_*-\frac{1}{2} \bar\sigma_*)\vert v \vert^2_{H^1(\Omega)}
\leq
a_{\bar\lambda}(v,v)
\leq (\bar\nu^*+\frac1 2 \bar\sigma^*) \vert v \vert^2_{H^1(\Omega)}
$$
for all $v\in H^1_0(\Omega)$, $\bar\lambda\in \bar\Lambda$. We conclude that setting 
$\alpha_*:=\bar\nu_*-\frac{1}{2} \bar\sigma_*=\frac 1 4 \nu_*$ and 
$\alpha^*= \bar\nu^*+\frac1 2 \bar\sigma^*= \nu^* + \frac 1 2 \sigma^* + \frac 3 4 \nu_*$
it holds 
\begin{equation} \label{eq:equiv}
\sqrt{\alpha_*} \vert v \vert_{H^1(\Omega)}\leq \tvert v\tvert_{\bar\lambda} \leq \sqrt{\alpha^*} \vert v \vert_{H^1(\Omega)}\qquad \forall v\in H^1_0(\Omega), ~\forall \bar\lambda\in \bar\Lambda
\end{equation}
with $\tvert v\tvert^2_{\bar\lambda}:=a_{\bar\lambda}(v,v)$. 
The space $\Lambda$ will be a subset of $\bar\Lambda$ containing the coefficients $\lambda$ of the problem \eqref{eq:main-problem}; it will be defined later on.

Concerning the definition of the space $F$ containing the right-hand side,
we write $f=(f_1,f_2)\in L^2(\Omega)\times L^2(\Omega)=:F$ (note that different couples in $F$ may give rise to the same $f\in H^{-1}(\Omega)$).

We now discuss the $hp$-discretization of \eqref{eq:two-point}. To this end, we specify that the binary master tree 
$\mathfrak{K}$ is obtained from an initial partition, called the `root partition', by 
applying successive dyadic subdivisions to all its elements. Later, cf. Property \ref{property:bound}, it will be needed to assume that this initial partition is sufficiently fine. Furthermore, with reference to the abstract notation of Section \ref{S1}, given any  $(K,d) \in \mathfrak{K}\times \mathbb{N}$ we have {$p(d)=d$. In consideration} of this simple relation, throughout this section 
we will use the notation $(K,p)$ instead of $(K,d)$, i.e., the second parameter of the couple will identify a polynomial degree on the element $K$.  
%\rsnote{Maybe good to spend some words about what the master tree $\mathfrak{K}$ is in this 1D setting. Something pops up about this issue in Property~\ref{property:bound}}
We set
%{\color{magenta} [cm: polynomial degrees to be discussed]}

\begin{eqnarray}
&&V_{K,p}=\mathbb{P}_p(K), \qquad F_{K,p}=\mathbb{P}_{p-1}(K)\times \mathbb{P}_p(K),\nonumber\\
&&\Lambda_{K,p}=\{\bar\lambda=(\bar\nu,\bar\sigma)\in  \mathbb{P}_{p+1}(K)\times 
\mathbb{P}_{p+1}(K):~~
\bar\nu_*\leq \bar\nu\leq \bar\nu^*,~ -\bar\sigma_*\leq \bar\sigma\leq \bar\sigma^*  \}.\nonumber 
\end{eqnarray}
Thus $${{V}}^c_\cD = \{ v \in H^1_0(\Omega) : v_{|K_D} \in \mathbb{P}_{p_D}(K_D) \ \forall D \in {\cD} \}$$ will be the discretization space associated with the $hp$-partition ${\mathcal D}$. Furthermore, we have 
$F_{\mathcal{D}} \subset F$ and $\Lambda_{\mathcal{D}} \subset \bar\Lambda$, with $F$ and $\bar\Lambda$ defined above. The difference in polynomial degrees between the various components of the approximation spaces for data is motivated by the need of balancing the different terms entering in the local error estimators, see  \eqref{def:e} below.

At this point,  we have all the ingredients that determine a Galerkin approximation as in \eqref{eq:galerkin}.

%------------------------------------------------------------------------------------
\subsection{Computable a posteriori error estimator}
%------------------------------------------------------------------------------------

%Let $\mathcal{D}=\{D=(K_D,p_D)\}$ be the {\em $hp$-discretization} at hand where an  element $D$ is a 
%pair of the element domain $K_D$, and the polynomial degree $p_D$. 
%We denote by $\#\cD:=\sum_{D\in\cD} p_D$ the cardinality of $\cD$, while $h_D = \text{diam}\, K_D$ defines the {\em diameter} of the element domain $K_D$. 

Given data $(f_\cD,\lambda_\cD)\in F_{\mathcal{D}} \times \Lambda_{\mathcal{D}}$, let $u_\cD(f_\cD,\lambda_\cD) \in V_\cD^c$ be the solution of the Galerkin problem \eqref{eq:galerkin} with such data. To it, we associate the residual $r=r(u_{\cD},f_\cD,\lambda_\cD) \in \HmOm$, defined by
\begin{equation}\label{res:def}
\langle r, v \rangle =  \langle { f}_{\cD} , v \rangle - a_{\lambda_\cD}(u_{\cD}(f_\cD,\lambda_\cD), v)
 \qquad \forall v \in \HOm \;,
\end{equation}
and satisfying  $\langle r, v_{\cD}  \rangle =0$  for all $v_{\cD} \in {V}_{\cD}^c$. The dual norm of the residual is a natural a posteriori error estimator, since one has
\begin{equation}\label{eq:apost-estim}
\frac1{\sqrt{\alpha^*}} \, \Vert r \Vert_{\HmOm} \leq  \tvert u(f_\cD,\lambda_\cD)-u_\cD(f_\cD,\lambda_\cD) \tvert_{\lambda_\cD} \leq
\frac1{\sqrt{\alpha_*}} \, \Vert r \Vert_{\HmOm} \;;
\end{equation}
in one dimension, such norm can be expressed in terms of independent contributions coming from the elements $K_D$ of the partition
$\cD$, which are easily and exactly computable if, e.g., the residual is locally polynomial. To see this,  
let us introduce the subspace of $\HOm$ of the piecewise linear functions on $\mathcal{D}$, i.e.,
$$
{{V}}^{L}_{\cD} = \{v \in \HOm \ | \ v_{|K_D} \in \mathbb{P}_1(K_D) \quad \forall  D \in \cD \} \subseteq V_\cD^c
$$
and let us first notice that $\HOm$ admits the orthogonal decomposition (with respect to the inner product associated with the norm $|{\cdot}|_{H^1(\Omega)}$)
$$
\HOm= { V}^{L}_{\cD}  \oplus \bigoplus_{D \in \cD} \HI \;,
$$
where functions in $\HI$ are assumed to be extended by $0$ outside the interval $K_D$;
indeed, for any $v \in V$, we have the orthogonal splitting
$$
v=v_L + \sum_{D \in \cD} v_{K_D} \;, 
$$
where $v_L \in  { V}^{L}_{\cD} $ is the piecewise linear interpolant of $v$ on $\cD$ and $v_{K_D}=(v-v_L)_{|K_D} \in \HI$. 
Recalling that $\langle r, v_L  \rangle =0$  for all $v_L \in {V}_{\cD}^L$, it is easily seen that the following expression holds:
\begin{equation} \label{eq:res-repr}
\Vert r \Vert_{\HmOm}^2 = \sum_{D \in \cD} \|r_{K_D}\|^2_{H^{-1}(K_D)} \;,
\end{equation}
where $r_{K_D}$ denotes the restriction of $r$ to $\HI$.

The computability of the terms on the right-hand side is assured by the following representation: for any $D \in \cD$, one has 
$$\|r_{K_D}\|^2_{H^{-1}(K_D)}=\vert z_{K_D}\vert^2_{H^1(K_D)},
%\footnote{\cm{We change $e_{K_D}$ into $z_{K_D}$ to avoid confusion with the local error functional introduced in Section $4.4$. Any better symbol is welcome}.} 
$$
where $ z_{K_D}\in \HI$ satisfies
\begin{equation}\label{eq:def-eK}
( z_{K_D}',v')_{L_2(K_D)}=\langle r_{K_D}, v \rangle \quad\forall v \in H^1_0(K_D).
\end{equation}
Writing $ u_\cD=u_\cD(f_\cD,\lambda_\cD)$ and $K_D=(a,b)$, and noting that, since $f_{2,\cD}$ is a polynomial in $K_D$, 
\begin{equation}\label{eq:res-split}
\langle r_{K_D}, v \rangle=\int_{K_D} \!\!\! \big(f_{1,\cD} + f_{2,\cD}'+(\nu_\cD u_\cD^\prime)^\prime - 
\sigma_\cD u_\cD \big)v \, dx= ( r_{K_D}, v )_{L_2(K_D)},
\end{equation} 
%, why don't we simply write 
%$\langle r_{K_D}, v \rangle=(r_{K_D}, v)_{L_2(K_D)}=\int_{K_D} \!\!\! (f_{1,\cD}+f_{2,\cD}'+(\nu_\cD u_\cD^\prime)^\prime - 
%\sigma_\cD u_\cD)v \, dx$, and write
%$e_{K_D}(x)=\int_{K_D} G(x,y) r_{K_D}(y) dy$? Then $r_1$ and $r_2$ don't have to be introduced.}
it is easily seen that the solution $ z_{K_D}$ has the following  analytic expression  
\begin{equation}\label{eq:green}
 z_{K_D}(x)=\int_{K_D} G(x,y) r_{K_D}(y) dy \;,
\end{equation}
where $G(x,y):=\left\{
\begin{array}{cc} 
\frac{(a-x)(b-y)}{b-a} & x<y\\
\frac{(a-y)(b-x)}{b-a} & x>y
\end{array}
\right.
$ is the Green's function of our local problem \eqref{eq:def-eK}. 
Thus, the squared norm $\|r_{K_D}\|^2_{H^{-1}(K_D)}$ of the local residual can be explicitly computed, since  $r_{K_D}$ is a polynomial. 
%This is precisely the case in our situation, since by definition $r_{1}$ and $r_{2}$ are polynomials. 
%as pointed out in the following simple result (see also the description of the module \Refine in Section \ref{sec:PDE}). \rsnote{Why do we need that?}
%\begin{property}\label{lm:res}
%Let $r^{(1)}$ and $r^{(2)}$ be polynomials of degree $p_1$ and $p_2$, respectively. Then  $e_{K_D}$ is a polynomial of degree $p:=\max(p_1+2,p_2+1)$. \endproof
%\end{property}

Summarizing, defining for any $D\in\cD$ the local error estimator
\begin{equation} \label{eq:lor-est}
\eta^2_{D,\cD}(u_\cD(f_\cD,\lambda_\cD),f_\cD,\lambda_\cD):= \vert  z_{K_D}\vert^2_{H^1(K_D)}
\end{equation}
%\rsnote{Elsewhere the local error indicator is denoted as $\eta_D(u_\cD,f_\cD,\lambda_\cD)$, and $\eta(u_\cD;\cD)$ as $\est_\cD(u_\cD,f_\cD,\lambda_\cD)$}
and defining the global error estimator as in \eqref{global-error-est}, we have by \eqref{eq:apost-estim}
\begin{equation}\label{eq:2nd-apost-estim}
\begin{split}
\frac1{\sqrt{\alpha^*}} \est_\cD(u_\cD(f_\cD,\lambda_\cD),f_\cD,\lambda_\cD) &\leq
\tvert u(f_\cD, \lambda_\cD)-u_\cD(f_\cD, \lambda_\cD) \tvert_{\lambda_\cD} \\
&\leq
\frac1{\sqrt{\alpha_*}}  \est_\cD(u_\cD(f_\cD,\lambda_\cD),f_\cD,\lambda_\cD),
\end{split}
\end{equation}
which in particular implies the reliability assumption \eqref{23}.

%------------------------------------------------------------------------------------
\subsection{The module \Refine}\label{sec:PDE}
%------------------------------------------------------------------------------------
Hereafter, we present a realization of the module \Refine, that 
guarantees the discrete efficiency property \eqref{24}, hence the contraction property of {\Reduce}. For every $D\in \cM\subseteq \cD$ the module raises the local polynomial degree to some higher value, whereas for $D\in \cD\setminus \cM$ the local polynomial degree remains unchanged. No $h$-refinement is performed.

To be precise, consider an element $D=(K_D,p_D)\in \cM$. Suppose that the local polynomial degree of the data is related to some $\hat{p}_{D}$, in the sense that $$f_{1,\cD}|_{K_D} \in \mathbb{P}_{\hat{p}_{D}-1}(K_D),\quad f_{2,\cD}|_{K_D} \in \mathbb{P}_{\hat{p}_{D}}(K_D),\quad \nu_\cD|_{K_D}, \ \sigma_\cD|_{K_D} \in \mathbb{P}_{\hat{p}_{D}+1}(K_D).$$ 
Recall that $u_\cD=u_\cD(f_\cD,\lambda_\cD)$ satisfies $u_\cD|_{K_D}\in \mathbb{P}_{p_D}(K_D)$. Then it is easily seen that the residual 
$r=r(u_\cD, f_\cD,\lambda_\cD)$ is such that its restriction $r_{K_D}$ to $K_D$ is a polynomial of degree 
$\hat{p}_D+ p_D+1$, while the function $ z_{K_D}$ defined in \eqref{eq:def-eK} is a polynomial of degree
\begin{equation}\label{eq:def-hatpD}
\bar{p}_D :=\hat{p}_D + p_D + 3\;.
\end{equation}
Therefore, the module \Refine builds  $ \bar \cD=\bar{\cD}(\cM)\in \mathbb{D}^c= \mathbb{D}$ with   $\bar{\cD}(\cM) \geq \cD$  as follows:
$$\bar\cD=\{\bar D\}\quad\text{with~} \bar D= 
\begin{cases}
(K_D,\bar p_D)& \text{for~}D\in \cM\\
D& \text{for~}D\in \cD\setminus\cM.
\end{cases}
$$

%Thus, for any $v \in H^1_0(K_D)$ we have $(e_{K_D}', v')_{L^2(K_D)} = (e_{K_D}', (\pi_1 v)')_{L^2(K_D)}$, where $\pi_1 v$ denotes
%the orthogonal projection in $H^1_0(K_D)$ of $v$ upon $\mathbb{P}_{\hat{p}_D}^0(K_D) := \mathbb{P}_{\hat{p}_D}(K_D) \cap H^1_0(K_D)$.
In order to prove  \eqref{24}, consider a marked element $D \in \cM$. Setting $\mathbb{P}_{\bar{p}_D}^0(K_D) := \mathbb{P}_{{\bar{p}_D}}(K_D) \cap H^1_0(K_D)$ and recalling that $  z_{K_D}\in \mathbb{P}_{\bar{p}_D}^0(K_D)$ we have 
\begin{equation}
\eta_{D,\cD}(u_\cD,f_\cD,\lambda_\cD) \, = \, |  z_{K_D} |_{H^1(K_D)} 
=\!\!\! \sup_{w \in \mathbb{P}_{\bar{p}_D}^0(K_D)} \!\!\!  \frac{( z_{K_D}', w')_{L^2(K_D)} }{| w |_{H^1(K_D)}}
%= \sup_{v \in  H^1_0(K_D)} \frac{(e_{K_D}', v')_{L^2(K_D)} }{| v |_{H^1(K_D)}} \leq
%\sup_{v \in  H^1_0(K_D)} \frac{(e_{K_D}', (\pi_1v)')_{L^2(K_D)} }{| \pi_1 v |_{H^1(K_D)}} \nonumber \\
%&=& \sup_{w \in \mathbb{P}_{\hat{p}_D}^0(K_D)} \frac{(e_{K_D}', w')_{L^2(K_D)} }{| w |_{H^1(K_D)}} 
=  \!\!\!  \sup_{w \in \mathbb{P}_{\bar{p}_D}^0(K_D)} \frac{\langle r_{K_D}, w \rangle }{| w |_{H^1(K_D)}} \;.
\end{equation}
%
%\rsnote{$| e_{K_D} |_{H^1(K_D)}=\sup_{w \in \mathbb{P}_{\hat{p}_D}^0(K_D)} \frac{(e_{K_D}', w')_{L^2(K_D)} }{| w |_{H^1(K_D)}} $ because $e_{K_D} \in  \mathbb{P}_{\hat{p}_D}^0(K_D)$. So there is no need to use $\pi_1$ here}
%
On the other hand, the Galerkin solution $u_{\bar\cD}=u_{\bar\cD}(f_\cD,\lambda_\cD)$ is such that its residual $\bar r =  r (u_{\bar\cD}{,f_\cD,\lambda_\cD})$ satisfies $\langle \bar{r}_{K_D}, w \rangle=0$ for all $w \in  \mathbb{P}_{\bar{p}_D}^0(K_D)$.
Thus, denoting by $a_{\lambda_\cD,K_D}(\cdot,\cdot)$ the restriction of the form $ a_{{\lambda_\cD}}(\cdot,\cdot)$ to $H^1(K_D)\times H^1(K_D)$,  and setting $\tvert v \tvert^2_{\lambda_\cD, K_D}= a_{{\lambda_\cD,K_D}}(v,v)$, we get
\begin{eqnarray}
\eta_{D,\cD}(u_\cD,f_\cD,\lambda_\cD) &=& \sup_{w \in \mathbb{P}_{\bar{p}_D}^0(K_D)} 
\frac{\langle r_{K_D}-\bar{r}_{K_D}, w \rangle }{| w |_{H^1(K_D)}}\nonumber\\ 
&=& \sup_{w \in \mathbb{P}_{\bar{p}_D}^0(K_D)} 
\frac{a_{\lambda_\cD,K_D}(u_\cD - u_{\bar\cD}, w)}{| w |_{H^1(K_D)}}
\leq \sqrt{\alpha^*} \tvert u_\cD - u_{\bar\cD} \tvert_{\lambda_\cD, K_D}\nonumber \;.
\end{eqnarray}
%
%\rsnote{Strictly speaking, the last inequality is't a consequence of \eqref{eq:equiv} or anything else that has been derived, but it relies on the same arguments that have been used to derive \eqref{eq:equiv}}
Squaring and summing-up over all $D \in \cM$, we obtain
\begin{equation}\label{discr-eff-2}
 {\est^2_\cD(\cM,u_{\cD},f_\cD,\lambda_\cD)}
  \leq \alpha^* \tvert  u_\cD - u_{\bar\cD} \tvert^2_{\lambda_\cD} \;,
\end{equation}
which immediately implies \eqref{24}.
 \begin{remark}
The choice of the error estimator and the refinement strategy indicated above guarantees that the reliability assumption \eqref{23} and the efficiency assumption \eqref{24} are fulfilled, hence the conclusions of Proposition \ref{prop_reduce} hold true. 
Actually, one can be more precise, since using \eqref{eq:2nd-apost-estim} and \eqref{discr-eff-2} and following the steps of the proof of Proposition 
 \ref{prop_reduce}, we get that the sequence of Galerkin approximations built by a call of \Reduce
 satisfies the contraction property \eqref{contraction} with contraction factor 
 $\kappa= \sqrt{1- \frac{\alpha_*}{\alpha^*}\theta}$. 
\end{remark}

%------------------------------------------------------------------------------------
\subsection{Convergence and optimality properties of \HPAFEM}\label{sec:HPAFEM-1D}
%------------------------------------------------------------------------------------
In this section we discuss the convergence and optimality properties of our adaptive algorithm \Hpafem in the present one-dimensional setting. To this end, we first specify the abstract functional framework introduced in Sect. \ref{S1}. We already set $V:=H^1_0(\Omega)$ and $F:=L^2(\Omega)\times L^2(\Omega)$. Concerning the space $\Lambda$ containing the coefficients of the operator, we assume 
stronger regularity than just $L^\infty(\Omega)$ in order to guarantee that the piecewise polynomial approximations of the coefficients still define a coercive variational problem. %\rsnote{This isn't the reason is it? I thought that the problem was that errors in $L^\infty$-norm aren't subadditive under $h$-refinements}. 

To be precise, from now on we assume that $\lambda=(\nu,\sigma)$ belongs to the space 
\begin{equation} \label{eq:lambda}
\Lambda:=\{ \lambda=(\nu,\sigma)\in H^1(\Omega) \times H^1(\Omega): ~~
\nu_*\leq \nu\leq \nu^*,~ 0\leq \sigma\leq \sigma^*  \}.
\end{equation}
%{\color{blue} Here, the choice of a smoothness space of Sobolev type with summability index $<+\infty$ guarantees that errors measured in the corresponding norm are subadditive under $h$-refinements.} \rsnote{This sentence is not really in the language that we have used so far. Maybe something like}
Here, in view of \eqref{1}, we choose to work with a smoothness space of Sobolev type with summability index $2$, so that squared best approximation errors are non-increasing under $h$-refinements.
We notice that it would be sufficient to require the coefficients to be piecewise $H^1$ on the initial partition. We decide to work under stronger assumptions just for the sake of simplicity.

We now define the projectors $Q_{K,p}$ introduced in Sect. \ref{def_ass}. To this end, let $\Pi^0_{K,p}\in \mathcal{L}(L^2(K),\mathbb{P}_p(K))$ be the $L^2$-orthogonal projection and  $\Pi^1_{K,p}\in \mathcal{L}(H^1(K),\mathbb{P}_p(K))$ be the $H^1$-type orthogonal projection defined as follows: if $v\in H^1(K)$ with $K=[a,b]$ then 
$$\left( \Pi^1_{K,p} v \right)(x) := c + \int_a^x  \left (\Pi^0_{K,p-1}  v^\prime \right )(t)\ dt$$ 
where the constant  $c$ is such that $\int_K  \Pi^1_{K,p} v \  dx = \int_K v \  dx$.

Then we define $Q_{K,p}\in\mathcal{L}(V\times F\times \Lambda, 
\mathbb{P}_p(K)\times (\mathbb{P}_{p-1}(K)\times \mathbb{P}_p(K))\times (\mathbb{P}_{p+1}(K)\times \mathbb{P}_{p+1}(K))$ by setting 
%{\color{magenta} [cm: to be discussed]}
$$ Q_{K,p} (v,f,\lambda):= (\Pi^1_{K,p} v_{\vert K},\ 
\Pi^0_{K,p-1} f_{1\vert K},\ \Pi^0_{K,p} f_{2\vert K},\  
\Pi^1_{K,p+1} {\nu}_{\vert K} ,\  \Pi^1_{K,p+1} {\sigma}_{\vert K}).$$
At last, we define the local error functionals $e_{K,p}$. We set 
\begin{equation}\label{def:e}
{e}_{K,p}(v,f,\lambda):= | ({\rm I} - \Pi^1_{K,p}) v_{\vert K} |^2_{H^1(K)} +  {\delta}^{-1}\text{osc}_{K,p}^2(f,\lambda)
\end{equation}
%\rsnote{I multiplied the definition with $\frac{1}{\delta}$ (which constant in the 2D case is called $M$, so this has to be synchronized). The advantage is that without data oscillation, it reduces to the usual notion of the error}
where $\delta>0$ is a positive penalization parameter to be chosen later 
and 
\begin{equation} \label{eq:oscillation}
\begin{split}
\text{osc}_{K,p}^2(f,\lambda)&:= \|\frac{h}{p}({\rm I}- \Pi^0_{K,p-1}) f_{1\vert K} \|^2_{L^2 (K)} +
\| ({\rm I}- \Pi^0_{K,p}) f_{2 \vert K} \|^2_{L^2 (K)} \\
 &+\  |({\rm I}- \Pi^1_{K,p+1}) \nu_{\vert K}|^2_{H^1(K)} 
+ | ({\rm I}- \Pi^1_{K,p+1}) \sigma_{\vert K}  |^2_{H^1(K)}
\end{split}
\end{equation}
where $h=\vert K\vert$. Note that the choice of polynomial degrees is such that for smooth data the four addends above scale in the same way with respect to the parameters $h$ and $p$. Furthermore, the data oscillation that appears in \eqref{def:e} is of higher order with respect to the projection error for the function $v$.

It is straightforward to check the validity of \eqref{1}.
%\rsnote{Maybe both here and in the 2D case, we should say something about the rational of the error functional and the selection of the polynomial approximation orders for the different functions. Indeed, since our main interest lies in finding a near-best approximation of the solution, it is desirable that, in any case for smooth forcing functions and parameters, the oscillation term is of higher order. This seems not to be the case for $\sigma$ and $\nu$, whereas to approximate $f_1$, polynomials of degree $p-1$ can be used}
We recall  that given a partition $\mathcal{D}\in \mathbb{D}$, we denote by $f_{\mathcal{D}}=(f_{1,\cD},f_{2,\cD})$ and 
$\lambda_{\mathcal D}=(\nu_\cD,\sigma_\cD)$ the piecewise polynomial function obtained by projecting $f$ and $\lambda$, respectively, element by element as indicated above. Note that while $f_\cD \in F_\cD \subset F$, $\lambda_\cD$ need not belong to $\bar{\Lambda}$. Given a partition $\mathcal{D}\in \mathbb{D}$, we will set  $${\rm osc}_\cD^2(f,\lambda):=\sum_{D \in \cD} {\rm osc}_D^2(f,\lambda),$$ where $ {\rm osc}_D^2(f,\lambda)=  {\rm osc}_{K_D,p_D}^2(f,\lambda)$.

The following result provides a uniform bound on the approximation error of the coefficients of the operator, assuring that $\lambda_\cD \in \bar{\Lambda}$.
%\rsnote{Why "Property" and not something common like "Proposition"?}
\begin{property}\label{property:bound}
Let $\hat{\mathcal{D}}$ be the root partition with polynomial degree equal to one on each element domain. 
Assume that  $\cK(\hat{\mathcal{D}})$ is sufficiently fine for the given data $\lambda \in \Lambda$, in the sense that
for each $K\in \mathcal{K}(\hat{\mathcal{D}})$ it holds
$$
 |({\rm I}- \Pi^1_{K,1}) \nu_{\vert K}|_{H^1(K)}\leq \frac{\nu_*}{2},\qquad 
 |({\rm I}- \Pi^1_{K,1}) \sigma_{\vert K}|_{H^1(K)}\leq \frac{\nu_*}{2}.
$$
Then for any $\mathcal{D}\in \mathbb{D}$ we have \eqref{eq:pert-bound}, i.e., 
$$
\| \nu - \nu_{\mathcal{D}}\|_{L^\infty(\Omega)}\leq \frac{\nu_*}{2},
 \qquad
\| \sigma - \sigma_{\mathcal{D}}\|_{L^\infty(\Omega)}\leq \frac{\nu_*}{2}. 
$$
Consequently, $\lambda_\cD \in \Lambda_{\mathcal{D}} \subset \bar\Lambda$.
\end{property}
\begin{proof}
For any $D=(K_D,p_D)$, let $\hat{K}\in \mathcal{K}(\hat{\cD})$ the element of the root partition containing $K_D$. 
%By the subadditivity property of the $H^1$ projection \rsnote{This would also be true for best approximations in $L_\infty$, so I suggest to skip this phrase}
Then, we have 
$$ |({\rm I}- \Pi^1_{K_D,p_D+1}) \nu_{\vert{K_D}}|_{H^1(K_D)}\leq  |({\rm I}- \Pi^1_{\hat{K},1}) \nu_{\vert {\hat{K}}}|_{H^1(\hat{K})}\leq \frac{\nu_*}{2}. $$
On the other hand, set $\psi=({\rm I}- \Pi^1_{K_D,p_D+1}) \nu\vert_{K_D}$; recalling that $\psi$ has zero mean-value in $K_D$, it vanishes at some point $x_0 \in K_D$ since it is a continuous function. Writing $\psi(x)=\psi(x_0)+\int_{x_0}^x \psi'(t)dt$ for any $x \in K_D$ yields
$$
|\psi(x)| \leq |x-x_0|^{1/2} \Vert \psi' \Vert_{L^2(K_D)} \leq |K_D|^{1/2} | \psi |_{H^1(K_D)}\;,
$$
whence the result immediately follows after observing $\vert K_D\vert \leq 1$.
\end{proof}

We now focus on the abstract assumptions \eqref{0}-\eqref{3}.
\begin{proposition}\label{prop:assumptions}
In the present setting, assumptions \eqref{2}-\eqref{3} hold true. Furthermore, if $\delta$ is chosen sufficiently small, then \eqref{0} is fulfilled. 
%\rsnote{Maybe it is nicer to formulate these results similarly to Propositions~\ref{prop10} and \ref{prop11}}.
\end{proposition}
\begin{proof}
{We start by verifying condition \eqref{3}. For any $v,w\in H^1_0(\Omega)$ and for any $\cD \in \mathbb{D}$ and any $D\in \cD$, it holds that 
\begin{eqnarray}
 |({\rm I}- \Pi^1_{K_D,p_D}) w_{\vert {K_D}}|_{H^1(K_D)} &=&  \inf_{\varphi \in \mathbb{P}_{p_D} (K_D)} 
 |w_{\vert {K_D}}-\varphi|_{H^1(K_D)}\nonumber\\   
 &\leq&   \inf_{\varphi \in \mathbb{P}_{p_D} (K_D)} |v_{\vert {K_D}}-\varphi|_{H^1(K_D)} +|(v-w)_{\vert {K_D}}|_{H^1(K_D)}\nonumber\\
 &=& |({\rm I}- \Pi^1_{K_D,p_D}) v_{\vert {K_D}}|_{H^1(K_D)} +|(v-w)_{\vert {K_D}}|_{H^1(K_D)}.\nonumber
\end{eqnarray}
Two applications of a triangle inequality show that 
\begin{align*}
&\Big| \err_\cD(v,f,\lambda)^{\frac12} -
\err_\cD(w,f,\lambda)^{\frac12} \Big| \\
&
\le 
\left( \sum_{D\in\cD} 
\left(\left(\big|({\rm I}- \Pi^1_D) v_{\vert {{K_D}}}\big|_{H^1({K_D})}^2+\delta^{-1}{\rm osc}^2_D(f,\lambda)\right)^{\frac 1 2}\right. \right.\\
&\qquad\qquad\qquad\qquad\qquad-\left.\left.\left(\big|({\rm I}- \Pi^1_D) v_{\vert {{K_D}}}\big|_{H^1({K_D})}^2+\delta^{-1}{\rm osc}^2_D(f,\lambda)\right)^{\frac 1 2}\right)^2
\right)^{\frac12} \\
&\leq
\left( \sum_{D\in\cD} 
\left(\big|({\rm I}- \Pi^1_D) v_{\vert {{K_D}}}\big|_{H^1({K_D})}-
\big|({\rm I}- \Pi^1_D) v_{\vert {{K_D}}}\big|_{H^1(K_D)}\right)^2\right)^{\frac12} \le \|v-w\|_V,
\end{align*}
i.e.,
\eqref{3} holds true with constant $C_2=1$.}

\vspace{0.5cm}
Let us now verify assumption \eqref{2}. Note that $u(f_{\cD},\lambda_\cD)$ is well defined since 
$\lambda_\cD\in \overline\Lambda$. Setting for simplicity $u=u(f,\lambda)$ and 
$\bar u=u(f_{\cD},\lambda_\cD)$, it is straightforward to check that $u-\bar u$ satisfies for any $v\in V$
\begin{eqnarray}
a_\lambda(u-\bar u, v)= \langle f - f_\cD , v \rangle - \int_\Omega (\nu - \nu_\cD) {\bar u}^\prime v^\prime \ dx
 - \int_\Omega (\sigma  - \sigma_\cD) {\bar u} v\ dx
\end{eqnarray}
whence, using the Poincar\'e inequality $\|v\|_{L^2(\Omega)}\leq 2^{-\frac 1 2} \vert v \vert_{H^1_0(\Omega)}$, and selecting $v=u-\bar{u}$, we obtain 
\begin{eqnarray}
\alpha_* | u-\bar{u} |_{H^1(\Omega)}& \leq& 
\| f_1- {f}_{1,\cD} \|_{H^{-1}(\Omega)} + 
 \| f_2- {f}_{2,\cD}\|_{L^2 (\Omega)} \nonumber\\
&&+ \left (\|\nu - \nu_\cD\|_{L^\infty(\Omega)}  +\frac 1 2\ \|\sigma - \sigma_\cD\|_{L^\infty(\Omega)}\right)
| \bar u |_{{H^1}(\Omega)}. \label{eq:pert}
\end{eqnarray} 
We now bound the quantity on the right hand side of \eqref{eq:pert} in terms of ${\rm osc}^2_\cD(f,\lambda)$.
To this end, starting with the first term, we have for any $v\in H^1_0(\Omega)$
\begin{eqnarray}
 (f_1 - f_{1,\cD},v)_{L^2(\Omega)} &=& \sum_{D\in \cD}( ({\rm I} - \Pi^0_{K_D,p_D-1}) f_{1\vert {K_D}}, v)_{L^2(K_D)}\nonumber\\
 &=& \sum_{D\in \cD}( ({\rm I} - \Pi^0_{K_D,p_D-1}) f_{1\vert {K_D}}, ({\rm I} - \Pi^0_{K_D,p_D-1}) 
 v_{\vert {K_D}}  )_{L^2(K_D)}\nonumber\\
&\leq&   \sum_{D\in \cD}  \| ({\rm I} - \Pi^0_{K_D,p_D-1}) f_{1 \vert {K_D}} \|_{L^2(K_D)}
\| ({\rm I} - \Pi^0_{K_D,p_D-1}) v_{\vert {K_D}} )\|_{L^2(K_D)}\nonumber\\
\end{eqnarray}
By the  classical $hp$-error estimate for the orthogonal $L^2$-projection upon $\mathbb{P}_{p_D}(K_D)$ 
(see, e.g., \cite[Corollary 3.12]{SCHW98} ) we have $\| ({\rm I} - \Pi^0_{K_D,p_D-1}) v_{\vert {K_D}} )\|_{L^2(K_D)}
\leq   \hat C \frac{h_\cD}{p_\cD}|v|_{H^1(K_D)}$ for some constant $\hat C >0$. Thus, we get 
\begin{equation}\label{2.6:1}
\| f_1- {f}_{1,\cD} \|_{H^{-1}(\Omega)}  \leq  \hat C \left( \sum_{D\in \cD} \| \frac{h_D}{p_D} ({\rm I} - \Pi^0_{K_D,p_D-1}) f_{1 \vert {K_D}}\|^2_{L^2(K_D)}\right)^{\frac 1 2}.
\end{equation}
Concerning the second term on the right hand side of \eqref{eq:pert}, we simply write it as 
\begin{equation}\label{2.6:2}
\| f_2- {f}_{2,\cD} \|_{L^2(\Omega)}  = \left( \sum_{D\in \cD} \| ({\rm I} - \Pi^0_{K_D,p_D}) f_{2\vert{K_D}}\|^2_{L^2(K_D)}\right)^{\frac 1 2}.
\end{equation}
Coming to the third and fourth terms, we first observe that 
\begin{equation}\label{2.6:3}
\begin{split}
| \bar u|_{H^1(\Omega)} &\leq \frac{1}{\alpha_*}\left( 2^{-\frac 1 2} \|f_{1,\cD}\|_{L^{2}(\Omega)} +  \|f_{2,\cD}\|_{L^2(\Omega)}  \right) \\
& \leq\frac{1}{\alpha_*} \left( 2^{-\frac 1 2} \| f_1\|_{L^{2}(\Omega)} +  \|f_2\|_{L^2(\Omega)} \right)=:C(f),
\end{split}
\end{equation}
since $f_{i,\cD}$, $i=1,2$ is locally an $L^2$-projection of $f_i$. 
%upon a polynomial space \rsnote{that it lives in a polynomial space is not relevant here}. 
On the other hand, using the same argument as in the proof of Property \ref{property:bound} we get 
\begin{eqnarray}
\| \nu - \nu_\cD\|_{L^\infty(\Omega)}&=&\max_{D\in\cD} \|({\rm I} - \Pi^1_{K_D,p_D+1}) \nu_{\vert {K_D}}\|_{L^\infty(K_D)} \nonumber\\
&\leq&  \max_{D\in\cD} | K_D|^{\frac 1 2} |({\rm I} - \Pi^1_{K_D,p_D+1}) \nu_{\vert {K_D}}|_{H^1(K_D)}
\nonumber\\
&\leq& \left (\sum_{D\in\cD}  |({\rm I} - \Pi^1_{K_D,p_D+1}) \nu_{\vert {K_D}}|^2_{H^1(K_D)}\right)^{\frac 1 2}.\label{2.6:4}
\end{eqnarray}
A similar result holds for $\| \sigma - \sigma_\cD\|_{L^\infty(\Omega)}$. Substituting \eqref{2.6:1}-\eqref{2.6:4} into \eqref{eq:pert} and recalling \eqref{eq:oscillation} we get 
\begin{eqnarray}
\alpha_*| u-\bar u|_{H^1(\Omega)} \leq \left( \frac 3 2 C(f) + \hat C +1\right) 
\left( \sum_{D \in \cD} {\rm osc}^2_D(f,\lambda)\right)^{\frac 1 2}.
\end{eqnarray}
Thus, setting
%\footnote{\cm{we removed $2$ from the numerator of $\bar C$.}} 
$\bar C:=\frac{1}{\alpha_*} \left (\frac 3 2 C(f) + \hat C +1\right )$ and recalling \eqref{def:e}, we conclude that 
\begin{equation}
|u(f,\lambda)- u(f_\cD,\lambda_\cD) |_{H^1(\Omega)} \leq \bar C \delta^{\frac 1 2} 
\left(  \sum_{D\in \cD} e_D(w,f,\lambda)\right)^{\frac 1 2}=  \bar C \delta^{\frac 1 2} E_\cD(w,f,\lambda)^{\frac 1 2}
\end{equation}
for any $w\in H^1_0(\Omega)$. This proves that \eqref{2} is fulfilled with $C_1=\bar C \delta^{\frac 1 2}$.
Finally, choosing any $\delta$ such that $C_1 < b$ we fulfill \eqref{0}.
\end{proof}

%\rsnote{Some attempt to help the reader:}\rs{
We conclude that choosing $\delta$ sufficiently small we may apply Theorem \ref{th1}.
This leads to the conclusion that for solving \eqref{eq:two-point}, where $f=(f_1,f_2) \in L^2(\Omega) \times L^2(\Omega)$, and $\lambda=(\nu,\sigma) \in \Lambda$ defined in  \eqref{eq:lambda}, and with a root partition $\hat{\cD}$ that is sufficiently fine such that it satisfies Property~\ref{property:bound},
\Hpafem is an instance optimal reducer, in the sense of Theorem~\ref{th1}, of the error functional
$$
E_\cD(u(f,\lambda),f,\lambda)=\sum_{D \in \cD} \inf_{\varphi \in \mathbb{P}_{p_D} (K_D)} |u(f,\lambda)_{|{K_D}}-\varphi|_{H^1(K_D)}^2+\delta^{-1} 
\text{osc}_{D}^2(f,\lambda),
$$
over all $\cD \in \mathbb{D}$, where $\text{osc}_{D}^2(f,\lambda)$ is defined in \eqref{eq:oscillation}.

Finally, we consider assumption \eqref{4}.  At first, we note that in one dimension all partitions are trivially conforming, i.e., {$\mathbb{D}^c= \mathbb{D}$. Next,} we observe that the following result holds.
\begin{lemma}  \label{lem3} 
For any $\cD\in \mathbb{D}$ and any $v\in H_0^1(\Omega)$ there holds
\begin{equation}
\inf_{w_\cD \in {V}^c_\cD} |v-w_\cD|_{H^1(\Omega)}^2 = \sum_{D \in \cD} \inf_{\varphi \in \mathbb{P}_{p_D}(K_D)} |v-\varphi|^2_{H^1(K_D)}.
\end{equation}
\end{lemma}

\begin{proof} For $D \in \cD$, let $q_D \in  \mathbb{P}_{p_D}(K_D)$ be such that $|v-q_D|_{H^1(K_D)}=\inf_{\varphi \in  \mathbb{P}_{p_D}(K_D)} |v-\varphi |_{H^1(K_D)}$.
Define $g \in L_2(\Omega)$ by $g|_{K_D}=q'_{D}$ for all $D \in \cD$, and $w_\cD \in H^1(\Omega)$ by $w_\cD(x)=\int_0^x g(s) ds$.
From $\int_{K_D} q'_D=\int_{K_D} v'$, we infer that $w_\cD(0)=w_\cD(1)=0$, and so $w_\cD \in {V}^c_\cD$. Moreover, $|v-w_\cD|_{H^1(\Omega)}^2= \sum_{D \in \cD}  |v-q_D|^2_{H^1(K_D)}$.
\end{proof}

Observing that $$ 
\inf_{\varphi \in \mathbb{P}_{p_D}(K_D)} |v-\varphi|^2_{H^1(K_D)}= | ({\rm I} - \Pi^1_{K_D,p_D}) v_{\vert {K_D}}|^2_{H^1(K_D)}\leq e_{D}(v,f,\lambda)$$ 
for any $f\in F$, $\lambda\in \Lambda$, we obtain the following result.
\begin{proposition} For all $\cD \in \mathbb{D}$ and all $v \in H^1_0(\Omega)$, one has
$$
\inf_{w_\cD \in {V}^c_\cD}  |v-w_\cD |_{H^1(\Omega)} \leq \inf_{(f,\lambda)
  \in F \times \Lambda}  \err_{\cD}(v,f,\lambda)^{\frac{1}{2}} ,
$$
i.e., {for $\cC:=I$} assumption \eqref{4} is fulfilled with $C_{3,\mathcal{D}}=1$. 
\end{proposition}
As a consequence, \eqref{18} and \eqref{19} are fulfilled with $C_{3}=C_4=1$.
Since {\Hpafem calls the routine \Reduce} with the fixed value 
$\varrho=\frac{\mu}{1+(C_1+1)\omega}$,  and by Proposition \ref{prop_reduce}
the number of iterations in \Reduce is bounded by $\mathcal{O}(\log \varrho^{-1})$, we are guaranteed that the number of iterations performed by \Reduce at any call from \Hpafem is uniformly bounded. On the other hand, recalling \eqref{eq:def-hatpD}, for each iteration in \Reduce the polynomial degree in each marked element is increased by a constant value depending only on the local polynomial degree in the input partition. Thus, even in the worst-case scenario that at each iteration all elements are marked for enrichment,  we conclude that the output partition of \Reduce has a cardinality which is bounded by a fixed multiple of the one of the input partition, which is optimal as it is produced by \Hpnearbest. 

Another obvious, but relevant application of Lemma~\ref{lem3} is that \Hpafem is an instance optimal reducer over $\cD \in \mathbb{D}$ of the error functional written in the more common form
$$
\inf_{w_\cD \in {V}^c_\cD} |u(f,\lambda)-w_\cD|_{H^1(\Omega)}^2 
+\delta^{-1} 
\text{osc}_{\cD}^2(f,\lambda).
$$

%%%%%%%%%%%%%%%%%%%%%%%%%%%%%%%%%%%%%%%%%%%%%%%%%%%%%%%%%%%%%%%%%%%%%%%%%%%%%%%%%%%%
\section{The Poisson problem in two dimensions}  \label{S3}
%%%%%%%%%%%%%%%%%%%%%%%%%%%%%%%%%%%%%%%%%%%%%%%%%%%%%%%%%%%%%%%%%%%%%%%%%%%%%%%%%%%%
On a polygonal domain $\Omega \subset \mathbb{R}^2$, we consider the Poisson problem
$$
\left\{ 
\begin{array}{r@{}c@{}ll}
-\triangle u &\,=\, & f & \text{in }\Omega,\\
u&\,=\,&0 &\text{on }\partial \Omega,
\end{array}
\right.
$$
in standard variational form. 
We consider right-hand sides $f \in L^2(\Omega)$, and so take $V=H^1_0(\Omega)$, $F=L^2(\Omega)$, and $\Lambda =\emptyset$.
We equip $H^1_0(\Omega)$ with $|\cdot|_{H^1(\Omega)}$, and $H^{-1}(\Omega)$ with the corresponding dual norm.

Let ${\cK}_0$ be an initial conforming triangulation of $\bar{\Omega}$, and let in each triangle in ${\cK}_0$ one of its vertices be selected as its newest vertex, in such a way that if an internal edge of the triangulation is opposite to the newest vertex of the triangle on one side of the edge, then it is also opposite to the newest vertex of the triangle on the other side. As shown in \cite[Lemma 2.1]{BDD04}, such an assignment of the newest vertices can always be made.

Now let $\mathbb{K}$ be the collection of all triangulations that can be constructed from ${\cK}_0$ by {{\it newest vertex bisection}, i.e., }a repetition of bisections of triangles by connecting their newest vertex by the midpoint of the opposite edge. With each bisection, two new triangles are generated, being `children' of the triangle that was just bisected, with their newest vertices being defined as the midpoint of the edge that has been cut. The set of all triangles that can be produced in this way is naturally organized as a binary master tree $\mathfrak{K}$, having as roots the triangles from ${\cK}_0$. The triangles from $\mathfrak{K}$ are uniformly shape regular.
The collection $\mathbb{K}$ of triangulations of $\Omega$ is equal to the sets of leaves of all possible subtrees of $\mathfrak{K}$.

For $K \in \mathfrak{K}$, we set $V_K=H^1(K)$ and $F_K=L^2(K)$, and for $d \in \mathbb{N}$, 
we set 
{
\begin{equation}\label{def:aux1}
V_{K,d} :=\mathbb{P}_{{p}(d)}(K),\quad
F_{K,d} :=\mathbb{P}_{{p}(d)-1}(K),
\end{equation}
with, as in Sect.~\ref{def_ass}, $p=p(d)$ being the largest value in $\mathbb{N}$ such that 
 $\text{dim} \, {\mathbb P}_{p-1}(K)={2+p-1 \choose p-1} \leq d$.
For example,  for $d=1,\ldots,10$,  we have $p=1,1,2,2,2,3,3,3,3,4$.}

\begin{remark} Alternatively, one can select sequences of strictly nested spaces $(V_{K,d})_d$, $(F_{K,d})_d$ with the condition that for the values of $d$ of the form {${2+p-1 \choose p-1}$ for some $p=:p(d) \in \mathbb{N}$}, definitions in \eqref{def:aux1} hold.
\end{remark}

For $D=(K_D,d_D) \in \cK \times \mathbb{N}$, we write $V_D=V_{K_D,d_D}$, $F_D=F_{K_D,d_D}$ and $p_D=p(d_D)$.
{Note that with the current definition of $V_D$, this space is uniquely determined by specifying $K_D$ and $p_D$.}
For some constant $\delta>0$ that will be determined later, we set the local error functional
$$
e_{D}(w,f):=e_D(w)+
\delta^{-1} \frac{|K|}{p_D^2} \inf_{f_{D} \in \mathbb{P}_{p_D -1}(K_D)}\|f-f_{D}\|_{L^2(K_D)}^2,
$$
where
\begin{equation} \label{210}
e_D(w):=\inf_{\{w_{D} \in \mathbb{P}_{p_D}(K_D)\colon \int_{K_D} w_{D}=\int_{K_D} w\}} |w-w_{D}|_{H^1(K_D)}^2.
\end{equation}
We define
\begin{equation} \label{211}
Q_D(w,f):=(w_D,f_D)
\end{equation}
as the pair of functions for which the infima are attained.

Having specified the master tree $\mathfrak{K}$, the local approximation spaces $V_D$ and $F_D$, the error functional $e_D(w,f)$, and the projection $Q_D(w,f)=(w_D,f_D)$, we have determined, according to  Sect. \ref{def_ass}, the collection of $hp$-partitions $\mathbb{D}$, the approximation spaces $V_\cD$ and $F_\cD$ for $\cD \in \mathbb{D}$, the global error functional 
  \begin{equation} \label{globalerr}
 \err_\cD(w,f)=\sum_{D \in \cD} e_D(w)+\delta^{-1} {\rm osc}_\cD^2(f),
 \end{equation}
 where
 $$
 {\rm osc}_\cD^2(f):=\sum_{D \in \cD} \frac{|K|}{p_D^2} \inf_{f_{D} \in \mathbb{P}_{p_D-1}(K_D)}\|f-f_{D}\|_{L^2(K_D)}^2,
 $$
 as well as the projection $f_\cD:=\prod_{D \in \cD} f_D$.
  \bigskip
 
 We proceed with verifying assumptions \eqref{0}, \eqref{2} and \eqref{3}.
 
 \begin{proposition} \label{prop100}
There holds
$$
 \sup_{f \in F } |\err_\cD(w,f)^{\frac{1}{2}}-\err_\cD(v,f)^{\frac{1}{2}}|  \leq  \|w-v\|_V \qquad \forall\cD \in \mathbb{D},~~\forall v,w \in V,
 $$
i.e., \eqref{3} is valid with $C_2=1$.
  \end{proposition}

\begin{proof}
{For $v,w \in V$, it holds that $e_D(w)^{\frac{1}{2}} \leq e_D(v)^{\frac{1}{2}}+|v-w|_{H^1(K_D)}$,
which yields the proof using the same arguments as in the proof of Proposition \ref{prop:assumptions}.}
%For $v,w \in V$, it holds that $e_D(w)^{\frac{1}{2}} \leq e_D(v)^{\frac{1}{2}}+|v-w|_{H^1(K_D)}$,
% and so
% \begin{align*}
%  \err_\cD(w,f)^{\frac{1}{2}}&=\left(\sum_{D \in \cD}e_D(w)+{\delta^{-1}} {\rm osc}_\cD^2(f)\right)^{\frac1 2}
%\\ & \leq \left(\sum_{D \in \cD}(e_D(v)^{\frac{1}{2}}+|w-v|_{H^1(K_D)})^2+{\delta^{-1}} {\rm osc}_\cD^2(f)\right)^{\frac1 2}
%\\
%& \leq \left(\sum_{D \in \cD}|w-v|_{H^1(K_D)}^2\right)^{\frac 1 2}+ \left(\sum_{D \in \cD}e_D(v)+{\delta^{-1}} {\rm osc}_\cD^2(f)\right)^{\frac 1 2}\\
%& =|w-v|_{H^1(\Omega)}+ \err_\cD(v,f)^{\frac{1}{2}},
%\end{align*}
%thanks to the same inequality used in the proof of Proposition \ref{prop:assumptions}.
\end{proof}

\begin{proposition} \label{prop11}
There holds
$$
|u(f)-u(f_\cD)|_{H^1(\Omega)} \lesssim \sqrt{\delta} \inf_{w \in H^1_0(\Omega)} \err_\cD(w,f)^{\frac{1}{2}} \quad\forall \cD \in \mathbb{D},~~\forall f\in L^2(\Omega),
$$
i.e., \eqref{2} is valid with $C_1 \eqsim \sqrt{\delta}$, and, when $\delta$ is chosen to be sufficiently small, so is \eqref{0}.
\end{proposition}

\begin{proof}
Since $f \mapsto u(f) \in \cL(H^{-1}(\Omega),H^1_0(\Omega))$ is an isomorphism, it is enough to estimate $\|f-f_\cD\|_{H^{-1}(\Omega)}$. To this end, we note that
for $K$ being a triangle and $p \in \mathbb{N}$, it holds that \cite{CHQZ06}
$$
\sup_{0 \neq w \in H^1(K)} \inf_{v \in \mathbb{P}_{p}(K)} \frac{\|w-v\|_{L^2(K)}}{|w|_{H^1(K)}} \lesssim\frac{\diam(K)}{p+1},
$$
only dependent on a lower bound for the smallest angle in $K$.
Consequently, we have that 
\begin{equation} \label{6}
\begin{split}
&\|f-f_\cD\|_{H^{-1}(\Omega)} =   \sup_{w \in H^1_0(\Omega)} \frac{ \inf_{v \in F_\cD}\langle f- f_\cD,w-v \rangle_{L^2(\Omega)}}{|w|_{H^1(\Omega)}}
\\
&\lesssim 
\sup_{w \in H^1_0(\Omega)} \frac{\sum_{D \in \cD} 
\frac{|K_D|^{\frac{1}{2}}}{p_D} \|f- f_D\|_{L^2(K_D)}  |w|_{H^1(K_D)}} 
{|w|_{H^1(\Omega)}}
\leq \sqrt{{\rm osc}_\cD^2(f)}.
\end{split}
\end{equation}
%The proof is completed by using that $f \mapsto u(f) \in \cL(H^{-1}(\Omega),H^1_0(\Omega))$ is an isomorphism.
\end{proof}

%------------------------------------------------------------------------------------
\subsection{{Conforming $h$-partitions, and conforming $hp$ finite
    element spaces}}\label{2D-conformity}
%------------------------------------------------------------------------------------
For the design of a routine \Reduce, in particular, for a posteriori error estimation, it is preferable to work with {$h$-partitions} that are conforming.
Let
$$
\mathbb{K}^c:=\{\cK \in \mathbb{K}\colon \cK \text{ is conforming}\}.
$$
As shown in \cite[Lemma 2.5]{BDD04}, for $\cK \in \mathbb{K}$, its smallest refinement
$\cK^c \in \mathbb{K}^c$ satisfies  $\# \cK^c \lesssim \#{\cK}$.

With the subclass 
$$
\mathbb{D}^c:=\{\cD \in \mathbb{D}\colon \cK(\cD) \in \mathbb{K}^c\},
$$
we define $\cC: \mathbb{D}\to\mathbb{D}^c$  
by setting $\cC(\cD)=\underline{\cD}$, where $\underline{\cD}$ is defined as the partition in  $\mathbb{D}^c$ with {\it minimal} $\# \underline{\cD}$ for which $\underline{\cD} \geq \cD$. That is, $\cK(\underline{\cD} )=\cK(\cD)^c$, and $p_{\underline{D}}=p_{D}$ for $\underline{D} \in \underline{\cD}$, $D \in \cD$ with $K_{\underline{D}} \subseteq K_{D}$.

Unfortunately, $\sup_{\cD \in \mathbb{D}} \frac{\# \cC(\cD)}{\# \cD} =\infty$, i.e., \eqref{19} is not valid. 
Indeed, as an example, consider $\cK_0$ to consist of two triangles $K_1$ and $K_2$.
Let $\cD \in \mathbb{D}$ be such that $K_1 \in \cK(\cD)$, with corresponding polynomial degree $p(d)$, and that in $\cK(\cD)$, $K_2$ has been replaced by $2^N$ triangles of generation $N$, each with polynomial degree $1$. Then $\# \cD \eqsim d+ 2^N$.
Since $\cK(\cC(\cD)) = \cK(\cD)^c$ contains in any case $\eqsim 2^{N/2}$ triangles inside $K_1$, so with polynomial degrees $p(d)$, we conclude that
$\# \cC(\cD) \gtrsim 2^N+2^{N/2} d$. By taking say $d \eqsim 2^N$, we conclude the above claim.

The fact that \eqref{19} does not hold implies that, unlike for an $h$-method, we will not have a proper control on the dimension of the finite element spaces that are created inside \Reduce.
\medskip

From \eqref{100}, recall the definition $V_{\cD}^c=V_{\cD} \cap H^1_0(\Omega)$ for $\cD \in \mathbb{D}^c$,
and from \eqref{210}-\eqref{211}, recall the definition of $e_D(w)$ and $w_D$ for $D \in \cD$ and $w \in H^1(K_D)$.
The main task in this section will be the proof of the following result.

\begin{theorem} \label{th2} Setting, for $\cD \in \mathbb{D}$, $\|p_\cD\|_\infty:=\max_{D \in \cD} p_D$ , 
for $\cD \in \mathbb{D}^c$ 
it holds that
$$
\inf_{v \in V_{\cD}^c} |w-v|_{H^1(\Omega)}^2 \lesssim (1+\log \|p_\cD\|_\infty)^3 \sum_{D \in \cD} e_D(w)\quad \forall w \in H^1_0(\Omega).
$$
\end{theorem}

Since, for $\cD \in \mathbb{D}$, obviously $\sum_{D \in \cC(\cD)} e_D(w) \leq \sum_{D \in \cD} e_D(w)$, Theorem~\ref{th2} implies \eqref{4} with
$$
C_{3,\cD} \eqsim (1+\log \|p_\cD\|_\infty)^\frac{3}{2}.
$$

For an underlying $h$-partition that is conforming, Theorem~\ref{th2} says that the error in $H^1$-norm of the best conforming $hp$-approximation of a $w \in H^1_0(\Omega)$, is at most slightly larger than 
the error in the broken $H^1$-norm of the best nonconforming $hp$-approximation.

The proof of this remarkable result will be based on Veeser's proof in \cite{Vee12} of the corresponding result in the `$h$'-setting.
In \cite{Vee12}, the result is shown by taking $v$ to be the Scott-Zhang (\cite{SZ90}) quasi-interpolant of $w$. 
This Scott-Zhang quasi-interpolation is constructed in terms of the nodal basis, and the proof relies on an inverse inequality applied to these basis functions, which inequality involves a multiplicative factor that is known to degrade seriously, i.e. not logarithmically, with increasing polynomial degree.

In our proof the role of the nodal basis on a triangle will be played by the union of the three linear nodal basis functions associated to the vertices, the polynomials on each edge that vanish at the endpoints, that will be boundedly extended to polynomials on the interior of the triangle, and, finally,  the polynomials on the triangle that vanish at its boundary.
We will construct a $\Pi_{\cD} \in \cL(H_0^1(\Omega),V_{\cD}^c)$ such that, with
\begin{equation} \label{700}
p_{D,\cD}:=\max_{\{D' \in \cD: K_{D'} \cap K_D \neq \emptyset\}} p_{D'} \quad \forall  D \in \cD,
\end{equation}
it holds that
\begin{equation} \label{701}
|(w-\Pi_{\cD} w)|_{K_D}|^2_{H^1(K_D)} \leq (1+\log p_{D,\cD})^3 \sum_{\{D' \in \cD: K_{D'} \cap K_D \neq \emptyset\}} e_{D'}(w) \quad \forall  D \in \cD,
\end{equation}
which obviously implies the statement of the theorem. Since the right-hand side of \eqref{701} vanishes for $w \in V_{\cD}^c$, because it even vanishes for $w \in V_{\cD}$, the mapping $\Pi_{\cD}$ is a {\it projector}.

\begin{proof}{({\it Theorem \ref{th2}})} Let $\cD \in \mathbb{D}^c$. In order to show \eqref{701}, it is sufficient to show
\begin{equation} \label{7}
|(\Pi_{\cD} w)|_{K_D}-w_D|^2_{H^1(K_D)} \leq (1+\log p_{D,\cD})^3 \sum_{\{D' \in \cD: K_{D'} \cap K_D \neq \emptyset\}} e_{D'}(w) \quad \forall  D \in \cD,\, w \in H_0^1(\Omega).
\end{equation}

Let $\cN(\cD)$ and $\cE(\cD)$ denote the collection of vertices (or nodes), and (closed) edges of $\cK(\cD)$.
To construct $\Pi_{\cD}$, for $e \in \cE(\cD)$ we set
\begin{equation} \label{15}
p_{e,\cD}:=\min\{p_D: D \in \cD,\, e \subset \partial K_D\},\quad \bar{p}_{e,\cD}:=\max\{p_D: D \in \cD,\, e \subset \partial K_D\}.
\end{equation}
With the mesh skeleton $\partial \cK(\cD):=\cup_{D \in \cD} \partial K_D$, we set
\begin{align*}
V_{\partial\cD}&:=\{v \in C(\partial \cK(\cD)): v|_e \in \mathbb{P}_{p_{e,\cD}}(e)\,\,\,\forall  e\in\cE(\cD)\},\\
\bar{V}_{\partial\cD}&:=\{v \in C(\partial \cK(\cD)): v|_e \in \mathbb{P}_{\bar{p}_{e,\cD}}(e)\,\,\,\forall  e\in\cE(\cD)\}.
\end{align*}
We construct $\Pi_{\partial\cD}\in \cL(\prod_{e \in \cE(\cD)} H^{\frac{1}{2}}(e),V_{\partial\cD})$, and an auxiliary $\bar{\Pi}_{\partial\cD}\in \cL(\prod_{e \in \cE(\cD)} H^{\frac{1}{2}}(e),\bar{V}_{\partial\cD})$, such that
\begin{equation} \label{9}
(\Pi_{\partial\cD} v)|_{\partial\Omega}=0 \quad \text{for all } v=(v_e)_{e \in \cE(\cD)} \in \prod_{e \in \cE(\cD)} H^{\frac{1}{2}}(e) \text{ with } v_e=0 \text{ for } e \subset \partial\Omega.
\end{equation}

For any triangle $K$ with edges $e_1,e_2,e_3$, there exists an extension $E_K \in \cL(H^{\frac{1}{2}}(\partial K),H^1(K))$ that, for any $p \in \mathbb{N}$, maps $C(\partial K) \cap \prod_{i=1}^3 \mathbb{P}_p(e_i)$ into $\mathbb{P}_p(K)$  (see e.g. \cite[Sect.7]{18.64}).
Defining $\Pi_{\cD}$ by
\begin{equation} \label{11} 
(\Pi_{\cD} w)|_{K_D}:=E_{K_D} ((\Pi_{\partial\cD} w|_{\partial \cK(\cD)})|_{\partial K_D})+w_D-E_{K_D}(w_D|_{\partial K_D}),
\end{equation}
in view of the definition of  $V_{\partial\cD}$ and \eqref{9}, we have $\Pi_{\cD} \in \cL(H_0^1(\Omega),V_{\cD}^c)$.

To construct $\Pi_{\partial\cD}$, $\bar{\Pi}_{\partial\cD}$, for each $\n \in \cN(\cD)$ we select some
\begin{equation} \label{10}
e_{\n}\in \cE(\cD) \text{ with } \n \in e_{\n}  \text{ and } e_{\n} \subset \partial\Omega  \text{ when } \n \in \partial\Omega.
\end{equation}
For $\n \in \cN(\cD)$, by $\phi_{\n}$ we denote the nodal hat function, i.e., $\phi_{\n}$ is continuous piecewise linear w.r.t. $\cK(\cD)$ and $\phi_{\n}(\hat{v})=\delta_{v,\hat{v}}$ $\forall v,\hat{v} \in \cN(\cD)$.
For $e \in \cE(\cD)$, let
\begin{align*}
\bar{Q}_e&:H^{\frac{1}{2}}(e) \rightarrow H^{\frac{1}{2}}(e) \text{ be the } H^{\frac{1}{2}}(e)\text{-orthogonal projector onto } \mathbb{P}_{\bar{p}_{e,\cD}}(e),\\
Q_{0,e}&:H^{\frac{1}{2}}(e) \rightarrow H^{\frac{1}{2}}(e) \text{ be the } H^{\frac{1}{2}}(e)\text{-orthogonal projector onto } \mathbb{P}_{p_{e,\cD}}(e) \cap H^1_0(e),\\
\bar{Q}_{0,e}&:H^{\frac{1}{2}}(e) \rightarrow H^{\frac{1}{2}}(e) \text{ be the } H^{\frac{1}{2}}(e)\text{-orthogonal projector onto } \mathbb{P}_{\bar{p}_{e,\cD}}(e) \cap H^1_0(e).
\end{align*}
Denoting the endpoints of an $e \in \cE(\cD)$ by $\n_{1,e},\n_{2,e}$, we now define $\Pi_{\partial\cD}$ and $\bar{\Pi}_{\partial\cD}$ by setting, for $v=(v_e)_{e \in \cE(\cD)} \in \prod_{e \in \cE(\cD)} H^{\frac{1}{2}}(e)$,
\begin{align*}
(\Pi_{\partial\cD} v)|_e&:=\sum_{i=1}^2 (\bar{Q}_{e_{\n_{i,e}}} v_{e_{\n_{i,e}}})(\n_{i,e})  \phi_{\n_{i,e}}|_e +Q_{0,e}\Big(v_e-\sum_{i=1}^2(\bar{Q}_{e_{\n_{i,e}}} v_{e_{\n_{i,e}}})(\n_{i,e})  \phi_{\n_{i,e}}|_e\Big),\\
(\bar{\Pi}_{\partial\cD} v)|_e&:=\sum_{i=1}^2 (\bar{Q}_{e_{\n_{i,e}}} v_{e_{\n_{i,e}}})(\n_{i,e})  \phi_{\n_{i,e}}|_e +\bar{Q}_{0,e}\Big(v_e-\sum_{i=1}^2(\bar{Q}_{e_{\n_{i,e}}} v_{e_{\n_{i,e}}})(\n_{i,e})  \phi_{\n_{i,e}}|_e\Big),
\end{align*}
for any $e \in \cE(\cD)$.
It is clear that $\Pi_{\partial\cD}$ maps into $V_{\partial\cD}$, and, thanks to \eqref{10}, that it satisfies \eqref{9}.
Similarly, $\bar{\Pi}_{\partial\cD}$ maps into $\bar{V}_{\partial\cD}$

These definitions show that, for $D \in \cD$, $(\Pi_{\cD} w)|_{K_D}$ depends only on $w|_{\cup\{K_{D'}: D' \in \cD,\,K_{D'} \cap K_D \neq \emptyset\}}$. Therefore, in order to prove \eqref{7}, a homogeneity argument shows that we may assume that $K_D$ is a uniformly shape regular triangle with
$$
|K_D|=1.
$$
Since the extension $E_{K_D} \in \cL(H^{\frac{1}{2}}(\partial K_D),H^1(K_D))$ can be chosen to be uniformly bounded over all such $K_D$, in view of \eqref{11} in order to arrive at \eqref{7}, and so at the statement of the theorem, what remains to be proven is that
\begin{equation} \label{e14}
\begin{split}
\|(\Pi_{\partial\cD} &w|_{\partial \cK(\cD)})|_{\partial K_D}-w_D|_{\partial K_D}\|^2_{H^{\frac{1}{2}}(\partial K_D)}\\
& \lesssim 
(1+\log p_{D,\cD})^{3} \sum_{\{D' \in \cD: K_{D'} \cap K_D \neq \emptyset\}}  e_{D'}(w) \quad \forall  D \in \cD,\,w \in H^1(\Omega).
\end{split}
\end{equation}

In \cite[Thms. 6.2 and 6.5]{18.64}, it was shown that on an interval $I$ of length $\eqsim 1$, it holds that
\begin{align} \label{200}
\|z\|_{L_\infty(I)} & \lesssim (1+\log p)^{\frac{1}{2}} \|z\|_{H^{\frac{1}{2}}(I)}\quad \forall  z \in \mathbb{P}_p(I),\\ \label{201}
\|z\|_{H^{\frac{1}{2}}_{00}(I)} & \lesssim (1+\log p) \|z\|_{H^{\frac{1}{2}}(I)}\quad \forall  z \in \mathbb{P}_p(I) \cap H^1_0(I).
\end{align}
These estimates will be used hereafter.
\begin{lemma} \label{203} For $\nu \in \cN(\cD) \cap \partial K_D$, $e,e'\in \cE(\cD)$ with $e \cap e' = \nu$, we have
$$
|(\bar{Q}_{e} w|_{e}-\bar{Q}_{e'} w|_{e'})(\nu)|^2 \lesssim \sum_{\{D'\in \cD\colon K_{D'} \ni \nu\}} (1+\log p_{D'}) e_{D'}(w) \qquad \forall w \in H^1(\Omega).
$$
\end{lemma}

\begin{proof} Consider the notations as in Figure~\ref{fig1}.
\begin{figure}
\begin{center}
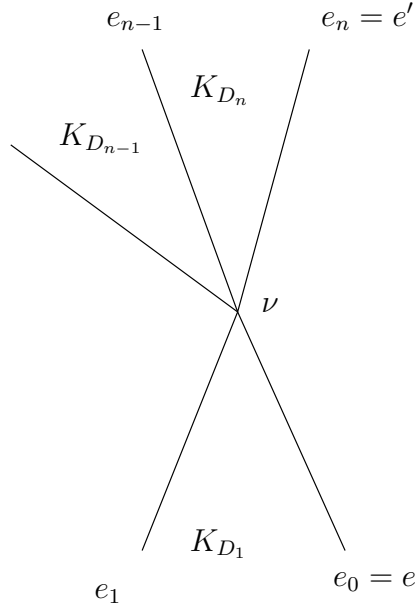
\end{center}
\caption{Notations relative to the proof of Lemma~\ref{203}}
\label{fig1}
\end{figure}
Using that for $1 \leq i \leq n$, $(\bar{Q}_{e_i} w_{D_i}|_{e_i})(\nu)=(\bar{Q}_{e_{i-1}} w_{D_i}|_{e_{i-1}})(\nu)$,
we have
\begin{align*}
&(\bar{Q}_{e_n} w|_{e_n}-\bar{Q}_{e_0} w|_{e_0})(\nu)=\\
&\Big(\bar{Q}_{e_n} (w-w_{K_{D_n}})|_{e_n}+
\sum_{i=1}^{n-1} \bar{Q}_{e_i} (w_{K_{D_{i+1}}}-w+w-w_{K_{D_{i}}})|_{e_i}+
\bar{Q}_{e_0} (w_{K_{D_1}}-w)|_{e_0}\Big)(\nu),
\end{align*}
and so, using \eqref{200} and the trace inequality,
$$
|(\bar{Q}_{e_n} w|_{e_n}-\bar{Q}_{e_0} w|_{e_0})(\nu)| \lesssim \sum_{i=1}^n\big[(1+\log \bar{p}_{e_{i-1}})^{\frac{1}{2}} +(1+\log \bar{p}_{e_{i}})^{\frac{1}{2}}\big] e_{D_i}(w)^{\frac{1}{2}}. 
$$
\end{proof}

We continue with the proof of Theorem~\ref{th2}.
As a first application of this lemma, we show that it suffices to prove \eqref{e14} with $\Pi_{\partial\cD}$ reading as $\bar{\Pi}_{\partial\cD}$. To this end, for $e \in \cE(\cD) \cap \partial K_D$, let $D' \in \cD$ such that $e \subset \partial K_{D'}$ and $p_{e,\cD}=p_{D'}$.
Then
\begin{align*}
(\Pi_{\partial\cD} w|_{\partial\cK(\cD)})|_e-(\bar{\Pi}_{\partial\cD} w|_{\partial\cK(\cD)})|_e=
(Q_{0,e}-\bar{Q}_{0,e})\Big(w|_e-\sum_{i=1}^2(\bar{Q}_{e_{\n_{i,e}}} w|_{e_{\n_{i,e}}})(\n_{i,e})  \phi_{\n_{i,e}}|_e\Big)\\
=
(Q_{0,e}-\bar{Q}_{0,e})\Big(w|_e-w_{D'}|_{e}-\sum_{i=1}^2(\bar{Q}_{e_{\n_{i,e}}} w|_{e_{\n_{i,e}}}-w_{D'}|_{e})(\n_{i,e}) \phi_{\n_{i,e}}|_e\Big).
\end{align*}
From \eqref{201}, the trace theorem and the property $\|  \phi_{\n_{i,e}}\|_{H^{\frac{1}{2}}(e)} \lesssim1$, we infer that
\begin{equation} \label{204}
\begin{split}
\|(\Pi_{\partial\cD} w|_{\partial\cK(\cD})|_e-&(\bar{\Pi}_{\partial\cD} w|_{\partial\cK(\cD})|_e\|_{H^{\frac{1}{2}}_{00}(e)}
\\
&\lesssim (1+\log \bar{p}_{e,\cD}) \Big(e_{D'}(w)^{\frac{1}{2}}+\sum_{i=1}^2 |(\bar{Q}_{e_{\n_{i,e}}} w|_{e_{\n_{i,e}}}-w_{D'}|_{e})(\n_{i,e})|\Big).
\end{split}
\end{equation}
Writing
$$
\bar{Q}_{e_{\n_{i,e}}} w|_{e_{\n_{i,e}}}-w_{D'}|_{e}=
\bar{Q}_{e_{\n_{i,e}}} w|_{e_{\n_{i,e}}}-\bar{Q}_e w|_e+\bar{Q}_e (w|_e-w_{D'}|_{e}),
$$
and applying \eqref{200} as well as the trace theorem, shows that
\begin{equation} \label{205}
\begin{split}
|(\bar{Q}_{e_{\n_{i,e}}} w|_{e_{\n_{i,e}}}-&w_{D'}|_{e})(\n_{i,e})| \\
&\lesssim 
|(\bar{Q}_{e_{\n_{i,e}}} w|_{e_{\n_{i,e}}}-\bar{Q}_e w|_e)(\n_{i,e})|+(1+\log(\bar{p}_{e,\cD}))^{\frac{1}{2}}e_{D'}(w)^{\frac{1}{2}}.
\end{split}
\end{equation}
By combining \eqref{204} and \eqref{205}, and applying Lemma~\ref{203} to the first term on the right-hand side of \eqref{205}, we conclude that
\begin{equation} \label{206}
\begin{split}
\|\big((\Pi_{\partial\cD}-&\bar{\Pi}_{\partial\cD}) w|_{\partial \cK(\cD)}\big)|_{\partial K_D}\|^2_{H^{\frac{1}{2}}(\partial K_D)}\\
& \lesssim 
(1+\log p_{D,\cD})^{3} \sum_{\{D' \in \cD: K_{D'} \cap K_D \neq \emptyset\}}  e_{D'}(w) \quad \forall  D \in \cD,\,w \in H^1(\Omega).
\end{split}
\end{equation}

As a consequence, what remains to show is \eqref{e14} with $\Pi_{\partial\cD}$ reading as $\bar{\Pi}_{\partial\cD}$, that is, to show that
\begin{equation} \label{212}
\begin{split}
\|(\bar{\Pi}_{\partial\cD} &w|_{\partial \cK(\cD)})|_{\partial K_D}-w_D|_{\partial K_D}\|^2_{H^{\frac{1}{2}}(\partial K_D)}\\
& \lesssim 
(1+\log p_{D,\cD})^{3} \sum_{\{D' \in \cD: K_{D'} \cap K_D \neq \emptyset\}}  e_{D'}(w) \quad \forall  D \in \cD,\,w \in H^1(\Omega).
\end{split}
\end{equation}

Let us first consider the situation that $e_{\n} \subset \partial {K_D}$ for all  $\n \in \cN(\cD) \cap \partial K_D$.
Then $((I-\bar{\Pi}_{\partial\cD}) w_D|_{\partial \cK(\cD)})|_{\partial K_D}=0$ (this is generally not true for $\Pi_{\partial\cD}$), and so
\begin{equation} \label{207}
\|(\bar{\Pi}_{\partial\cD} w|_{\partial \cK(\cD)})|_{\partial K_D}- w_D|_{\partial K_D}\|_{H^{\frac{1}{2}}(\partial K_D)} =\|(\bar{\Pi}_{\partial\cD}(w-w_D)|_{\partial K_D})|_{\partial K_D}\|_{H^{\frac{1}{2}}(\partial K_D)}.
\end{equation}

To bound the right-hand side, let us write $v=(w-w_D)|_{\partial K_D}$.
For edges $e_1,e_2$ of $\partial K_D$, and $\n:=e_1 \cap e_2$, an application of \eqref{200} shows that
\begin{equation} \label{208}
\|(\bar{Q}_{e_{\n}} v_{e_{\n}})(\n)  \phi_{\n}\|_{H^{\frac{1}{2}}(\partial K_D))} 
\lesssim (1+\log \bar{p}_{e_{\n},\cD})^{\frac{1}{2}} \|v_{e_{\n}}\|_{H^{\frac{1}{2}}(e_{\n})}.
\end{equation}

For an edge $e \subset \partial K_D$, applications of \eqref{201} and \eqref{200} show that
\begin{equation} \label{209}
\begin{split}
&\|\bar{Q}_{0,e}\Big(v|_e-\sum_{i=1}^2(\bar{Q}_{e_{\n_{e,i}}} v|_{e_{\n_{e,i}}})(\n_{e,i})  \phi_{\n_{e,i}}|_e\Big)\|_{H_{00}^{\frac{1}{2}}(e)}
\\& \lesssim (1+\log \bar{p}_{e,\cD})
\|\bar{Q}_{0,e}\Big(v|_e-\sum_{i=1}^2(\bar{Q}_{e_{\n_{e,i}}}v|_{e_{\n_{e,i}}})(\n_{e,i})  \phi_{\n_{e,i}}|_e\Big)\|_{H^{\frac{1}{2}}(e)} \\
&\lesssim (1+\log \bar{p}_{e,\cD})
\Big(\|v|_{e}\|_{H^{\frac{1}{2}}(e)}+\max_{i=1,2}(1+\log \bar{p}_{e_{\n_{e,i}},\cD})^{\frac{1}{2}} \|v|_{e_{\n_{e,i}}}\|_{H^{\frac{1}{2}}(e_{\n_{e,i}})}\Big).
\end{split}
\end{equation}
Combination of \eqref{207}, \eqref{208}, and \eqref{209}, together with an application of the trace theorem, show that, in the situation of $e_{\n} \subset \partial {K_D}$ for all  $\n \in \cN(\cD) \cap \partial K_D$,
$$
\|(\bar{\Pi}_{\partial\cD} w|_{\partial K_D})|_{\partial K_D}-w_D|_{\partial K_D}\|^2_{H^{\frac{1}{2}}(\partial K_D)}\\
 \lesssim 
(1+\log p_{D,\cD})^{3}  e_{D}(w) \quad \forall  w \in H^1(K_D),
$$
which implies \eqref{212}.

Consider now the situation that for one (or similarly more) $\n \in \cN(\cD) \cap \partial K_D$, $e_{\n} \not\subset \partial K_D$.
We estimate the difference, in $H^{\frac{1}{2}}(\partial K_D)$-norm, with the situation that $e_{\n}$ is equal to some edge $\bar{e}  \subset \partial K_D$. 
Applications of  \eqref{201} and Lemma~\ref{203} show that 
\begin{equation*}
\begin{split}
\|\sum_{\{e \in \cE(\cD) \cap \partial K_D: e \ni \n\}} &(I-\bar{Q}_{0,e})\big((\bar{Q}_{e_{\n}} w|_{e_{\n}}-\bar{Q}_{\bar{e}} w|_{\bar{e}})(\n) \phi_{\n}|_e\big)\|_{H^{\frac{1}{2}}(\partial K_D)}\\
&\lesssim 
\big(1 +\sum_{\{e \in \cE(\cD) \cap \partial K_D: e \ni \n\}} (1+\log \bar{p}_{e,\cD})\big) |(\bar{Q}_{e_{\n}} w|_{e_{\n}}-\bar{Q}_{\bar{e}} w|_{\bar{e}})(\n)|\\
& \lesssim 
(1+\log p_{D,\cD})^{\frac{3}{2}} \sum_{\{D' \in \cD: K_{D'} \cap K_D \neq \emptyset\}}  e_{D'}(w)^{\frac{1}{2}},
\end{split}
\end{equation*}
which completes the proof of \eqref{212}, and thus of the theorem.
\end{proof}

%------------------------------------------------------------------------------------
\subsection{The routine {\Reduce}} \label{SubPDE}
%------------------------------------------------------------------------------------
For $\cD \in {\mathbb{D}}^c$, with $w_\cD$ we will denote the best approximation to $w$ from $V^c_\cD=V_\cD \cap H^1_0(\Omega)$ w.r.t. $|\cdot|_{H^1(\Omega)}$. For $w=u(f)$, being the solution of the Poisson problem with right-hand side $f$, $u_\cD(f)$ turns out to be the Galerkin approximation to $u(f)$ from $V^c_\cD$.

In this section, we will apply results from \cite{MW01} on residual based a posteriori error estimators in the $hp$ context.
These results were derived under the condition that the polynomial degrees $p_D$ and $p_D'$ for $D,D' \in \cD \in \mathbb{D}^c$ with $K_D \cap K_{D'} \neq \emptyset$ differ not more than an arbitrary, but constant factor.
Fixing such a factor, let $\breve{\mathbb{D}}^c$ denote the subset of those $\cD \in {\mathbb{D}}^c$ that satisfy this condition. Obviously, for each $\cD \in {\mathbb{D}}^c$, there exists a $\breve{\cD} \in \breve{\mathbb{D}}^c$ with $\cK(\breve{\cD})=\cK(\cD)$ and $\breve{\cD} \geq \cD$. Unfortunately, even for the smallest possible of such $\breve{\cD}$, let us write it as $\breve{\cD}(\cD)$, the ratio $\# \breve{\cD}(\cD) / \# \cD$ cannot be bounded uniformly in $\cD \in {\mathbb{D}}^c$.

In view of the replacement of ${\mathbb{D}}^c$ by $\breve{\mathbb{D}}^c$, 
the mapping $\cC: \mathbb{D} \rightarrow \mathbb{D}^c$ constructed in the previous subsection has to be replaced by
$\breve{\cC}:=\mathbb{D} \rightarrow \breve{\mathbb{D}}^c:\cD \mapsto \breve{\cD}(\cC(\cD))$.
From now on, we will denote $\breve{\mathbb{D}}^c$ as ${\mathbb{D}}^c$, and $\breve{\cC}$ as $\cC$.
Since obviously $\breve{\cD}$ can be constructed such that $\|p_{\breve{\cD}}\|_\infty=\|p_{\cD}\|_\infty$, with these new definitions \eqref{4} is still valid with $C_{3,\cD} \eqsim (1+\log(\|p_\cD\|_\infty))^{\frac{3}{2}}$, and, as before, unfortunately $\sup_{\cD \in \mathbb{D}} \# \cC(\cD) / \# \cD=\infty$.

\medskip 

{We note that in the present application}, for $\cD \in \mathbb{D}^c$, $f_\cD \in F_\cD$, and a desired reduction factor $\varrho\in (0,1]$, 
$\Reduce(\varrho,\cD,f_{\cD})$ has to produce a $\cD \leq \bar{\cD} \in \mathbb{D}^c$ such that $|u(f_\cD)-u_{\bar{\cD}}(f_\cD)|_{H^1(\Omega)} \leq \varrho |u(f_\cD)-u_{\cD}(f_\cD)|_{H^1(\Omega)}$.
As explained in Section~\ref{cost}, {the $i$-th iteration of \HPAFEM performs a call} of $\Reduce(\frac{\mu}{1+(C_1+C_{3,\cD_i})\omega},\cC(\cD_i),f_{\cD_i})$. The scalars $\mu$ and $\omega$ are parameters as set in \HPAFEM. They depend on the constant $b$ from \Hpnearbest, cf.~Sect.\ref{S:2.2}, the constant $C_2$, here being equal to $1$, cf. Proposition~\ref{prop100}, and the constant $C_1$, here being $\eqsim \sqrt{\delta}$, see Proposition~\ref{prop11}. The scalar $\delta$ is a parameter in the definition of the error functional $E$, see \eqref{globalerr}, that is chosen such that $C_1 C_2<b$, cf. \eqref{0}. 
The only possible dependence of the required reduction factor  $\frac{\mu}{1+(C_1+C_{3,\cD_i})\omega}$ on $\cD_i$ is via the value of $C_{3,\cD_i}$.
As we have seen, $C_{3,\cD_i}\eqsim (1+\log \|\bar{p}_{\cD_i}\|_\infty)^{\frac{3}{2}}$, meaning that when the maximum polynomial degree in $\cD_i$ tends to infinity, this reduction factor tends to zero, but only very slowly.
 \medskip

The construction of $\Reduce$ will follow the general template given in Sect.~\ref{reduce}. We will verify the assumptions \eqref{260}, \eqref{27}, and \eqref{28}.
For $(w,f) \in H^1_0(\Omega) \times L^2(\Omega)$, $\cD \in \mathbb{D}^c$, and $D \in \cD$, we set the residual based (squared) a posteriori error indicator
$$
\eta^2_{D,\cD}(w,f):=\frac{|K_D|}{p_D^2}\|f+\triangle w\|_{L^2(K_D)}^2+\sum_{\{e \in  \cE(\cD): e \subset \partial K_D \cap \Omega\}} \frac{|e|}{2 p_{e,\cD}} \|\llbracket \nabla w \cdot {\bf n}_e\rrbracket\|_{L^2(e)}^2,
$$
where $p_{e,\cD}$ is from \eqref{15}, and define
$$
\est_\cD(w,f_\cD):=\left(\sum_{D \in \cD} \eta^2_{D,\cD}(w,f_\cD)\right)^{1/2}.
$$

The following theorem stems from \cite[Theorem~3.6]{MW01}. Inspection of the proof given therein shows that the local lower bound provided by the (squared) a posteriori error indicator applies to any $w \in V^c_\cD$ and so not only to the Galerkin solution.

\begin{theorem}[`reliability and efficiency'] \label{th10} There exists a constant $R>0$ such that for $\cD \in \mathbb{D}^c$ and $f_\cD \in F_\cD$,
\begin{equation} \label{14}
|u(f_\cD)-u_\cD(f_\cD)|^2_{H^1(\Omega)} \leq R \est_\cD^{2}(u_\cD(f_\cD),f_\cD).
\end{equation}
For any $\eps>0$, and all $\cD \in \mathbb{D}^c$, there exists an $r_{\cD,\eps} \eqsim \|p_\cD\|_{\infty}^{-2-2\eps}$, such that for all $f_\cD \in F_\cD$, and $w \in V^c_\cD$,
\begin{equation} \label{17}
r_{\cD,\eps}  \est_\cD^{2}(w,f_\cD)\leq  |u(f_\cD)-w|^2_{H^1(\Omega)}.
\end{equation}
\end{theorem}

\begin{corollary}[`stability'] \label{corol1} With $r_{\cD,\eps}$ as in Theorem~\ref{th10}, for all $\cD \in \mathbb{D}^c$, $f_\cD \in F_\cD$, and $v, w \in V^c_\cD$, it holds that
\begin{equation} \label{180}
\sqrt{r_{\cD,\eps}}  \Big|{\est_\cD(v,f_\cD)}-{\est_\cD(w,f_\cD)}\Big |\leq  |v-w|_{H^1(\Omega)}.
\end{equation}
\end{corollary}

\begin{proof} A repeated application of the triangle inequality, first in $\ell_2$ sequence spaces and then in $L^2$ function spaces, shows that
\begin{align*}
& |{\est_\cD(v,f_\cD)}-{\est_\cD(w,f_\cD)}| \\
&\leq 
\left(\sum_{D \in \cD} \frac{|K_D|}{p_D^2}\|\triangle (v-w)\|_{L^2(K_D)}^2+\sum_{\{e \in  \cE(\cD): e \subset \partial K_D \cap \Omega\}} \frac{|e|}{2 p_{e,\cD}} \|\llbracket \nabla (v-w) \cdot {\bf n}_e\rrbracket\|_{L^2(e)}^2\right)^{\frac12}\\
& \leq r_{\cD,\eps}^{-\frac{1}{2}}  |v-w|_{H^1(\Omega)},
\end{align*}
where the last inequality follows from an application of \eqref{17} with ``$f_\cD$'' reading as $0$, and ``$w$'' reading as $v-w$.
\end{proof}

What is left is to establish the `estimator reduction by refinement', i.e. \eqref{28}.
Given $\cM \subset \cD \in \mathbb{D}^c$, we define $\bar{\cD}(\cM) \in \mathbb{D}^c$ as follows:
The $h$-partition $\cK(\bar{\cD}(\cM))$ is the smallest in $\mathbb{K}^c$ in which each $K_D$ for $D \in \cM$ has been replaced by its four grandchildren in $\mathfrak{K}$; and for $D \in \bar{\cD}(\cM)$, it holds that $p_D=p_{D'}$ where $D' \in \cD$ is such that $K_{D'}$ be either equal to $K_D$, or its ancestor in $\cK(\cD)$.

\begin{proposition}[`estimator reduction by refinement'] \label{lem1} For $\cM \subset \cD \in \mathbb{D}^c$, and $\bar{\cD}(\cM) \in \mathbb{D}^c$ defined above, it holds that $\# \bar{\cD}(\cM) \lesssim \# \cD$.
For any $f_\cD \in F_{\cD}$, the estimator reduction property \eqref{28} is valid for $\gamma=\frac{1}{2}$.
\end{proposition}

\begin{proof} This follows easily from the fact that each $K_D$ ($D \in \cM$) is subdivided into four subtriangles that have equal area, that each $e \in \cE(\cM)$ is cut into two equal parts, and that the jump of the normal derivative of $w \in V^c_{\cD}$ over a newly created edge, i.e., an edge interior to a $K_D$ for $D \in \cD$, is zero.
\end{proof}

Given $\cD \in  \mathbb{D}^c$ and $f \in F_\cD$, let $\cD =\cD_0 \leq \cD_1 \leq \cdots \subset \mathbb{D}^c$ be the sequence of $hp$-partitions produced by $\Reduce(\varrho,\cD,f_\cD)$.
We have established \eqref{14}, \eqref{17}, and \eqref{180}, for any fixed $\eps>0$, as well as Proposition~\ref{lem1}.
Observing that $\|p_{\bar{\cD}(\cM)}\|_\infty=\|p_{\cD}\|_\infty$,
an application of Proposition~\ref{prop12} now shows that in each iteration the quantity
$$
|u(f_\cD)-u_{\cD_i}(f_\cD)|_{H^1(\Omega)}^2+(1-\sqrt{\bar{\gamma}})r_{\cD,\eps} \est_{\cD_i}(u_{\cD_i},f_\cD),
$$
where $\bar{\gamma}=(1-\theta)+\theta/2$,
is reduced by at least a factor 
$1-\frac{(1-\sqrt{\bar{\gamma}})^2}{2}\frac{r_{\cD,\eps}}{R}$, and that this quantity is equivalent to $|u(f_\cD)-u_{\cD_i}(f_\cD)|_{H^1(\Omega)}^2$. In view of $r_{\cD,\eps} \eqsim \|p_\cD\|_{\infty}^{-2-2\eps}$, we conclude that in order
to reduce the initial error $|u(f_\cD)-u_{\cD}(f_\cD)|_{H^1(\Omega)}$ by a factor $\varrho$ by an application of \Reduce, the number of iterations that are required is
%\footnote{\cm{We use the symbol $\eqsim$ because it seems to us that there were hidden constants not taken into account.}}
$$
M \eqsim \log(1/\varrho) \|p_\cD\|_{\infty}^{2+2\eps}.
$$

\begin{remark} 
This result is not satisfactory because the number of iterations grows more than quadratically with the maximal polynomial degree. Yet, it improves upon the result stated in \cite{BaPaSa:14}, where the number of iterations scales with the fifth power of the maximal polynomial degree.
\end{remark}

\medskip 
\noindent {\bf Acknowledgement}\\   
\noindent The first and the fourth authors are partially supported by GNCS-INdAM and the Italian research grant  {\sl Prin 2012}  2012HBLYE4  ``Metodologie innovative nella modellistica differenziale numerica''. The second author is partially supported by NSF grants DMS-1109325 and DMS-1411808.

%\end{acknowledgement}

%\bibliographystyle{alpha}
%\bibliography{../ref}

\def\cprime{$'$}

 \end{document}